\documentclass[a4paper,11pt]{article}
\usepackage{amsmath,amsthm,amssymb}
\usepackage[mathscr]{eucal}

\numberwithin{equation}{section}

{\theoremstyle{definition}\newtheorem{definition}{Definition}[section]

\newtheorem{remark}[definition]{Remark}
\newtheorem{example}[definition]{Example}}

\newtheorem{proposition}[definition]{Proposition}
\newtheorem{lemma}[definition]{Lemma}
\newtheorem{theorem}[definition]{Theorem}

\newcommand{\GL}{\operatorname{GL}}

\newcommand{\M}{\operatorname{M}}

\newcommand{\rH}{\operatorname{H}}

\newcommand{\cR}{\mathcal{R}}
\newcommand{\acts}{\curvearrowright}
\newcommand{\actson}[1][]{\overset{#1}{\curvearrowright}}
\newcommand{\SL}{\operatorname{SL}}
\newcommand{\rL}{\operatorname{L}}
\newcommand{\Aut}{\operatorname{Aut}}
\newcommand{\Out}{\operatorname{Out}}
\newcommand{\N}{\mathbb{N}}
\newcommand{\T}{\mathbb{T}}
\newcommand{\Z}{\mathbb{Z}}
\newcommand{\cF}{\mathcal{F}}

\newcommand{\cV}{\mathcal{V}}
\newcommand{\id}{\mathord{\operatorname{id}}}
\newcommand{\si}{\sigma}

\newcommand{\recht}{\rightarrow}
\newcommand{\cU}{\mathcal{U}}
\newcommand{\vphi}{\varphi}
\newcommand{\cW}{\mathcal{W}}
\newcommand{\R}{\mathbb{R}}
\newcommand{\al}{\alpha}
\newcommand{\eps}{\varepsilon}

\newcommand{\ovt}{\overline{\otimes}}
\newcommand{\B}{\operatorname{B}}

\newcommand{\om}{\omega}

\newcommand{\Q}{\mathbb{Q}}

\newcommand{\Ker}{\operatorname{Ker}}

\newcommand{\Gammatil}{\widetilde{\Gamma}}

\newcommand{\Rp}{\R_+}

\newcommand{\ot}{\otimes}

\newcommand{\cG}{\mathcal{G}}

\newcommand{\be}{\beta}
\newcommand{\ox}{\overline{x}}
\newcommand{\rO}{\operatorname{O}}
\newcommand{\rP}{\operatorname{P}}
\newcommand{\Ufin}{\mathcal{U}_{\text{\rm fin}}}
\newcommand{\og}{\overline{g}}
\newcommand{\PSL}{\operatorname{PSL}}
\newcommand{\PGL}{\operatorname{PGL}}
\newcommand{\lspan}{\operatorname{span}}

\newcommand{\Xtil}{\widetilde{X}}
\newcommand{\cN}{\mathcal{N}}
\newcommand{\trans}{^\top}

\begin{document}
\begin{center}
{\LARGE\bf Cocycle and orbit superrigidity  for lattices \vspace{0.5ex} \\ in
$\SL(n,\R)$ acting on homogeneous spaces}

\bigskip

{\sc by Sorin Popa\footnote{Partially supported by NSF Grant
DMS-0601082}\footnote{Mathematics Department; University of
    California at Los Angeles, CA 90095-1555 (United States).
    \\ E-mail: popa@math.ucla.edu} and Stefaan Vaes\footnote{Partially
    supported by ERC Starting Grant VNALG-200749 and Research
    Programme G.0231.07 of the Research Foundation --
    Flanders (FWO)}\footnote{Department of Mathematics;
    K.U.Leuven; Celestijnenlaan 200B; B--3001 Leuven (Belgium).
    \\ E-mail: stefaan.vaes@wis.kuleuven.be}}
\end{center}

\begin{abstract}\noindent
We prove cocycle and orbit equivalence superrigidity for lattices in
$\SL(n,\R)$ acting linearly on $\R^n$, as well as acting
projectively on certain flag manifolds, including the real
projective space. The proof combines operator algebraic techniques
with the property~(T) in the sense of Zimmer for the action
$\SL(n,\Z) \actson \R^n$, $n \geq 4$. We also show that the
restriction of the orbit equivalence relation $\cR(\SL(n,\Z) \actson
\R^n)$ to a subset of finite Lebesgue measure, provides a II$_1$
equivalence relation with property (T) and yet fundamental group
equal to $\R_+$.
\end{abstract}

\section{Introduction and statement of main results}

Over the last few years, operator algebraic methods were used to
prove several orbit equivalence and cocycle superrigidity theorems:
for Bernoulli actions of property (T) groups \cite{Pcocycle} and of
product groups \cite{Pproduct} and for profinite actions of property
(T) groups \cite{ioana}. In this paper, we extend the scope of these
methods to a more geometric class of actions, like the natural
actions of lattices $\Gamma \subset \SL(n,\R)$ on the vector space
$\R^n$, on the projective space $\rP^{n-1}(\R)$ and on certain flag
manifolds, all of which can be viewed as $\SL(n,\R)$-homogeneous
spaces.

None of these actions is probability measure preserving.
Hence, property (T) of the acting group, has to be replaced
by Zimmer's notion of property (T) for a non-singular action
(see \cite{Z1}), which plays a crucial role in this paper.
It is shown that for any lattice $\Gamma \subset \SL(n,\R)$,
the linear action $\Gamma \actson \R^n$ has property (T)
if and only if $n \geq 4$. We then deduce the following theorem.

\begin{theorem}
Let $\Gamma < \SL(n,\R)$ be a lattice and let $\cR$ be the II$_1$
equivalence relation obtained by restricting
the orbit equivalence relation $\cR(\Gamma \actson \R^n)$ to a
set of Lebesgue measure $1$, for some $n\geq 4$. Then we have:
\begin{itemize}
\item
$\cR$ has property (T), in the sense of Zimmer, yet the fundamental
group of $\cR$ equals $\R_+$.
\item
$\cR^t$ cannot be implemented by a free action of a group, $\forall
t>0$. Also, $\cR^t$ cannot be implemented by an (not necessarily
free) action of a discrete property (T) group, $\forall t>0$.
\end{itemize}
\end{theorem}

We say that a Polish group is of finite type if it can be realized
as the closed subgroup of the unitary group of some II$_1$ factor
with separable predual. All countable and all second countable
compact groups are Polish groups of finite type. In \cite{Pcocycle},
the first author proved that every $1$-cocycle for the Bernoulli
action of a property (T) group with values in a Polish group of
finite type, is cohomologous to a group morphism. We say that
actions with this property are \emph{$\Ufin$-cocycle superrigid.} More precisely:

\begin{definition}
The non-singular action $G \actson (X,\mu)$ of the locally compact
second countable group $G$ on the standard measure space $(X,\mu)$
is called \emph{$\Ufin$-cocycle superrigid} if every $1$-cocycle
for the action $G \actson (X,\mu)$ with values in a Polish group
of finite type $\cG$ is cohomologous to a continuous group morphism $G \recht \cG$.
\end{definition}

The following actions are known to be $\Ufin$-cocycle superrigid:
Bernoulli actions of property (T) groups \cite{Pcocycle} and of
product groups \cite{Pproduct}, while in \cite{ioana}, virtual
$\Ufin$-cocycle superrigidity is proven for profinite actions
of property (T) groups. We extend this to the following actions of geometric nature.

\begin{theorem}\label{thm.cocycle-superrigid-concrete}
The following actions are $\Ufin$-cocycle superrigid.
\begin{enumerate}
\item For $n \geq 5$ and $\Gamma$ any lattice in $\SL(n,\R)$,
the linear action $\Gamma \actson \R^n$.
\item For $n \geq 5$ and $\Gamma$ any finite index subgroup of $\SL(n,\Z)$,
the affine action $\Gamma \ltimes \Z^n \actson \R^n$.
\item For $n \geq 4k + 1$, $\Gamma$ any lattice in $\SL(n,\R)$ and $H$
any closed subgroup of $\GL(k,\R)$, the action
$$G \actson \M_{n,k}(\R) \quad\text{where}\quad G := \begin{cases}
\frac{\Gamma \times H}{\{\pm (1,1)\}} &\quad\text{if $(-1,-1) \in \Gamma \times H$,} \\
\Gamma \times H &\quad\text{otherwise,}\end{cases}$$
by left-right multiplication on the space $\M_{n,k}(\R)$ of $n \times k$
matrices equipped with the Lebesgue measure.
\end{enumerate}
\end{theorem}

In \cite{Pcocycle}, the first author introduces the notion of malleability of
a measure preserving action $\Gamma \actson (X,\mu)$, which roughly means that
there is a flow on $X \times X$, commuting with the diagonal $\Gamma$-action
and connecting the identity map to the flip map on $X \times X$. Theorem 0.1
in \cite{Pcocycle} says that every weakly mixing, malleable action of a
property (T) group is $\Ufin$-cocycle superrigid.

We generalize this cocycle superrigidity theorem to infinite measure preserving actions.
But, property (T) of the group $\Gamma$ has to be replaced by property (T) for the
diagonal action of $\Gamma \actson X \times X$. In the case of $\SL(n,\Z) \actson \R^n$,
this forces $n \geq 5$. Finally, weak mixing has to be replaced by the ergodicity of the
$4$-fold diagonal action $\Gamma \actson X \times X \times X \times X$, which in the case
of $\SL(n,\Z) \actson \R^n$ again holds exactly for $n \geq 5$.

Using the cocycle superrigidity of $\SL(n,\Z) \actson \R^n$, we give a full classification
of all $1$-cocycles for the action $\SL(n,\Z) \actson \T^n = \R^n/\Z^n$, with values in a
Polish group of finite type. As such, our Example \ref{ex.cocycles-torus} below,
complements Zimmer's celebrated cocycle superrigidity theorem \cite{Z2}~: Zimmer's
result treats arbitrary actions $\SL(n,\Z) \actson (X,\mu)$, but specific target
groups (simple linear algebraic groups), while our result treats a very specific
action, but rather general target groups.

Given the cocycle superrigidity theorem \ref{thm.cocycle-superrigid-concrete},
we can deduce several orbit equivalence (OE) superrigidity results.
We are particularly interested in the following concrete actions of
lattices $\Gamma$ in $\SL(n,\R)$ and $\PSL(n,\R)$.

\begin{enumerate}
\item The linear action $\Gamma \actson \R^n$.
\item If $\Gamma$ is a finite index subgroup of $\SL(n,\Z)$, the action
$\Gamma \actson \T^n = \R^n / \Z^n$.
\item The projective action $\Gamma \actson \rP^{n-1}(\R)$.
\item Let $X$ be the real flag manifold of signature $(d_1,\ldots,d_l,n)$.
Recall that points in $X$ are flags
$$\{0\} \subset V_1 \subset \cdots \subset V_l \subset \R^n$$
where $V_i$ is a vector subspace of $\R^n$ with dimension $d_i$. We consider
the natural action $\Gamma \actson X$ for any lattice $\Gamma$ in $\PSL(n,\R)$.
\end{enumerate}

The action in 1 has the Lebesgue measure as infinite invariant measure,
while the actions in 3 and 4 do not have finite or infinite invariant measures.
All the actions in 1-4 are essentially free and ergodic, see
Lemma \ref{lemma.measures-and-co} for details.

The natural invariant measure class on the flag manifold $X$ can
be described as follows. Put $d_l = k$ and consider the set $\M_{n,k}(\R)$
of $n \times k$ matrices of rank $k$, equipped with the Lebesgue measure.
Denote by $E=(E_1,\ldots,E_l)$ the standard flag of signature $(d_1,\ldots,d_l,n)$,
i.e.\ $E_i = \lspan\{e_1,\ldots,e_{d_i}\}$, where $e_1,\ldots,e_n$ are the standard
basis vectors in $\R^n$.
The group $\GL(k,\R)$ acts on $\M_{n,k}(\R)$ by right multiplication.
This action is free and proper and
$$\M_{n,k}(\R) / H \recht X : A \mapsto (AE_1,\ldots,AE_l)$$
is an isomorphism. Here, $H = \{g \in \GL(k,\R) \mid g E_i \subset E_i \;\;\text{for all}\;\;
i=1,\ldots,l\}$. Writing $k_1 = d_1$ and $k_i = d_i - d_{i-1}$ for $i \geq 2$, the group $H$
can of course be written as
\begin{equation} \label{eq.formula}
H = \begin{pmatrix} \GL(k_1,\R) & * & \cdots & * \\
0 & \GL(k_2,\R) & \cdots & * \\
\vdots & \vdots & \ddots & \vdots \\
0 & 0 & \cdots & \GL(k_l,\R) \end{pmatrix} \; .
\end{equation}

Before stating our OE superrigidity results, recall the following terminology.

\begin{definition}
Let $\Gamma \actson[\al] (X,\mu)$ and $\Lambda \actson[\be] (Y,\eta)$ be essentially
free, ergodic, non-singular actions of countable groups on standard measure spaces.
\begin{itemize}
\item A \emph{stable orbit equivalence (SOE)} between $\al$ and $\be$ is a
non-singular isomorphism $\Delta : X_0 \recht Y_0$ between non-negligible
subsets $X_0 \subset X$, $Y_0 \subset Y$, such that $\Delta$ is an isomorphism
between the restricted orbit equivalence relations $\cR(\Gamma \actson X)|_{X_0}$
and $\cR(\Lambda \actson Y)|_{Y_0}$.
\item We say that $\Gamma \actson X$ is \emph{induced} from $\Gamma_1 \actson X_1$,
if $\Gamma_1$ is a subgroup of $\Gamma$, $X_1$ is a non-negligible subset of $X$
and $g \cdot X_1 \cap X_1$ is negligible for all $g \in \Gamma - \Gamma_1$.
\end{itemize}
\end{definition}

For the linear lattice actions and the quotient action $\SL(n,\Z) \actson \T^n$,
we get the following.

\begin{theorem}\label{thm.OEsuperrigid-1}
Let $n \geq 5$ and $\Gamma \subset \SL(n,\R)$ a lattice. Let $\Lambda \actson
(Y,\eta)$ be any essentially free, ergodic, non-singular action of the countable
group $\Lambda$.
\begin{enumerate}
\item\label{item.1} The actions $\Gamma \actson \R^n$ and $\Lambda \actson Y$ are
SOE if and only if $\Lambda \actson Y$ is conjugate to an induction of one of the
following actions~:
    \begin{itemize}
    \item $\Gamma \actson \R^n$ itself,
    \item (only in case $-1 \in \Gamma$) the quotient action $\Gamma/\{\pm 1\}
    \actson \R^n/\{\pm 1\}$.
    \end{itemize}
\item\label{item.1bis}\footnote{In fact, this statement is a slightly more
detailed version, with very different proof, of \cite[Corollary B]{F1}, where
it is shown that for all $n \geq 3$, the actions $\SL(n,\Z) \actson \T^n$ and
$\Lambda \actson Y$ are SOE if and only if they are virtually conjugate.}
The actions $\SL(n,\Z) \actson \T^n$ and $\Lambda \actson Y$ are SOE if and
only if $\Lambda \actson Y$ is conjugate to an induction of one of the following actions~
    \begin{itemize}
    \item $\SL(n,\Z) \actson \T^n$ itself,
    \item $\SL(n,\Z) \ltimes \Z^n \actson \R^n$,
    \item $\displaystyle \SL(n,\Z) \ltimes \Bigl(\frac{\Z}{\lambda \Z}\Bigr)^n
    \actson \frac{\R^n}{\lambda \Z^n}$ for some $\lambda \in \N \setminus \{0,1\}$,
    \item (only in case $n$ is even) one of the actions $$\PSL(n,\Z) \actson
    \frac{\T^n}{\{\pm 1\}} \quad\text{or}\quad \PSL(n,\Z) \ltimes \Bigl(\frac{\Z}{2\Z}\Bigr)^n
    \actson \frac{\R^n / (2\Z)^n}{\{\pm 1\}} \; .$$
    \end{itemize}
\end{enumerate}
\end{theorem}

To formulate easily the correct OE superrigidity statements for lattice actions on flag
manifolds, make the following observations.

The real flag manifold of signature $(d_1,\ldots,d_l,n)$ has a natural $2^l$-fold
covering $\Xtil$ consisting of oriented flags
$$\{0\} \subset (V_1,\om_1) \subset \cdots \subset (V_l,\om_l) \subset \R^n$$
where every $V_i$ is a vector subspace of $\R^n$ with an orientation $\om_i$.
Clearly, $\Xtil = \M_{n,k}(\R) / H_0$, where $H_0 = \{g \in H \mid \det\bigl(g|_{E_i}\bigr) > 0
\;\;\text{for all}\;\;i=1,\ldots,l \}$. In the expression \eqref{eq.formula} above, $H_0$
consists of those matrices $A$ that have on the diagonal $A_{ii} \in \GL(k_i,\R)$ with
$\det A_{ii} > 0$ for all $i$.

Denote $\Sigma_l = H/H_0$ and observe that $\Sigma_l \cong \bigl(\Z/2\Z \bigr)^{\oplus l}$.
Then, $\Sigma_l$ acts on $\Xtil$ by reversing orientations, but keeping the flags.
We denote by $-1 \in \Sigma_l$ the multiplication by $-1$ and observe that $-1 = 1$
in $\Sigma_l$ if all $d_i$ are even. Clearly, $X = \Xtil / \Sigma_l$.

\begin{theorem}\label{thm.OEsuperrigid-2}
Let $X$ be the real flag manifold of signature $(d_1,\ldots,d_l,n)$. Let $\Gamma
\subset \PSL(n,\R)$ be a lattice and assume that $n \geq 4d_l +1$. Denote by $\Xtil$
the $2^l$-fold covering of $X$ consisting of oriented flags, as explained before the
theorem. Let $\Lambda \actson (Y,\eta)$ be any essentially free, ergodic, non-singular
action of the countable group $\Lambda$.

The actions $\Gamma \actson X$ and $\Lambda \actson Y$ are SOE if and only if
$\Lambda \actson Y$ is conjugate to an induction of one of the actions
    \begin{itemize}
    \item $\displaystyle \Gamma \times \frac{\Sigma_l}{\Sigma} \actson
    \Xtil/\Sigma \quad\text{for some subgroup}\;\; \Sigma < \Sigma_l \;\;\text{with}\;\;
    -1 \in \Sigma \; ,$
    \item $\displaystyle \frac{\Gammatil \times \Sigma_l/\Sigma}{\{\pm (1,1)\}}
    \actson \Xtil/\Sigma \quad\text{for some}\;\; \Sigma < \Sigma_l \;\;\text{with}\;\;
    -1 \not\in \Sigma \;$ and with $\Gammatil = \{\pm 1\} \cdot \Gamma \subset \GL(n,\R)$.
    \end{itemize}
\end{theorem}

\begin{example}
If $n \geq 5$ and $\Gamma \subset \PSL(n,\R)$ is a lattice, the action $\Gamma \actson
\rP^{n-1}(\R)$ is a special case of the flag manifold action treated in Theorem
\ref{thm.OEsuperrigid-2}. Hence, $\Gamma \actson \R^n$ and $\Lambda \actson Y$
are SOE if and only if $\Lambda \actson Y$ is conjugate to an induction of either
$\Gamma \actson \rP^{n-1}(\R)$ or its double cover $\Gammatil \actson \R^n/\R_+$,
where $\Gammatil = \{\pm 1\} \cdot \Gamma \subset \GL(n,\R)$.
\end{example}

Finally, combining the work of \cite{F2} and the above OE superrigidity results,
we classify up to stable orbit equivalence, the lattice actions on $\R^n$ and on
flag manifolds, see Theorems \ref{thm.classif-linear} and \ref{thm.classif-flag}.
At the same time, we compute the outer automorphism group of the associated orbit
equivalence relation.

\section{Preliminaries}

We recall here Zimmer's definition of property (T) for a II$_1$
equivalence relation $\cR$ on a standard probability space $(X,\mu)$.

To this end, first define $\cR^{(2)} = \{(x,y,z) \in X \times X
\times X \mid x \cR y \;\text{and}\; y \cR z \}$. Note that $\cR$,
resp.\ $\cR^{(2)}$ come equipped with canonical $\sigma$-finite
measures $\mu^{(1)}$, resp.\ $\mu^{(2)}$, given by
\begin{align*}
\mu^{(1)}(Y) &= \int_X \#\{ y \in X \mid (x,y) \in Y \} \; d\mu(x) =
\int_X \# \{ x \in X \mid (x,y) \in Y \} \; d\mu(y) \; ,\\
\mu^{(2)}(Y) &= \int_X \#\{ (y,z) \in \cR \mid (x,y,z) \in Y \} \; d\mu(x) =
\int_X \#\{ (x,z) \in \cR \mid (x,y,z) \in Y \} \; d\mu(y) \\
& = \int_X \#\{ (x,y) \in \cR \mid (x,y,z) \in Y \} \; d\mu(z) \; .
\end{align*}

\begin{itemize}
\item A \emph{$1$-cocycle of $\cR$ with values in the unitary group $\cU(H)$}
of a Hilbert space $K$ is a Borel map $c : \cR \recht \cU(K)$ satisfying
$c(x,z) = c(x,y) c(y,z)$ for almost all $(x,y,z) \in \cR^{(2)}$.
\item Suppose that $c : \cR \recht \cU(K)$ is a $1$-cocycle of $\cR$.
\begin{itemize}
\item A \emph{unit invariant vector} of $c$ is a Borel map $\xi : X \recht K$
satisfying $\xi(x) = c(x,y) \xi(y)$ for almost all $(x,y) \in \cR$ and $\|\xi(x)\|=1$
for almost all $x \in X$.
\item A \emph{sequence of almost invariant unit vectors} of $c$ is a sequence
of Borel maps $\xi_n : X \recht K$ satisfying
$$\|\xi_n(x) - c(x,y) \xi_n(y)\| \recht 0 \quad\text{for almost all}\quad (x,y) \in \cR$$
and $\|\xi_n(x)\|=1$ for all $n \in \N$ and almost all $x \in X$.
\end{itemize}
\end{itemize}

\begin{definition} \label{def.Tequiv}
A II$_1$ equivalence relation $\cR$ is said to have \emph{property (T) in the
sense of Zimmer} if the following holds: every $1$-cocycle of $\cR$ with values
in the unitary group of a Hilbert space and admitting a sequence of almost
invariant unit vectors, admits a unit invariant vector.
\end{definition}

\section{Property (T) for actions of locally compact groups}

We recalled above Zimmer's definition of property (T) for a II$_1$ equivalence relation.
In fact, one can define property (T) for measured groupoids in general,
see \cite{claire}. We do not need this generality in this paper,
but we do need the concept of property (T) for non-singular actions of locally compact second countable (l.c.s.c.)
groups on measure spaces. For a groupoid approach to this definition, we refer to
\cite{claire}. For the convenience of the reader, we gather in this section the necessary
concepts and results and present them in an operator algebra framework.

All von Neumann algebras are supposed to have separable predual and all locally
compact groups are supposed to be second countable.

If $M$ is a von Neumann algebra, we equip $\Aut(M)$ with the Polish topology
making the functions $\Aut(M) \recht M_* : \al \mapsto \om \circ \al$ continuous
for all $\om \in M_*$. An action $\al$ of a l.c.s.c.\ group $G$ on a von
Neumann algebra $M$, denoted $G \actson[\al] M$, is a continuous group morphism
$\al : G \recht \Aut(M)$.

\begin{itemize}
\item A \emph{$1$-cocycle} of an action $G \actson[\al] M$ with values in the
unitary group $\cU(K)$ of a Hilbert space $K$, is a strongly continuous map
$c : G \recht \cU(M \ovt \B(K))$ satisfying $c(gh) = c(g) (\al_g \ot \id)(c(h))$
for all $g,h \in G$. Note that by Theorem 3 in \cite{Moore3}, it makes no
difference to assume only that $c$ is a measurable map, with the previous
formula holding for almost all $(g,h) \in G \times G$.
\item A \emph{unit invariant vector} of the $1$-cocycle $c$ of $G \actson[\al] M$,
is an element $\xi$ in the $W^*$-module $M \ovt K$ satisfying $\xi^* \xi = 1$
and $c(g) (\al_g \ot \id)(\xi) = \xi$ for all $g \in G$.
\item A \emph{sequence of almost invariant unit vectors} of the $1$-cocycle $c$
of $G \actson[\al] M$, is a sequence $\xi_n \in M \ovt K$ satisfying $\xi_n^* \xi_n = 1$
for all $n$ and $c(g) (\al_g \ot \id)(\xi_n) - \xi_n \recht 0$ $^*$-strongly, uniformly
on compact subsets of $G$.
\item The action $G \actson M$ is said to have \emph{property (T)} if every $1$-cocycle
with values in the unitary group of a Hilbert space and admitting a sequence of
almost invariant unit vectors, admits a unit invariant vector.
\end{itemize}

The following result is proven for discrete groups in \cite[Proposition 2.4]{Z1},
see also \cite[Corollary 5.16]{claire}. These methods work as well in the locally
compact case and for completeness, we give a proof in a von Neumann algebra setup.

\begin{proposition}\label{prop.Tpmp}
Let $G \actson[\al] M$ be an action of the l.c.s.c.\ group $G$ on the von
Neumann algebra $M$. Suppose that $\tau$ is a faithful normal tracial state on $M$,
invariant under $\al$. Then, $G \actson M$ has property (T) if and only if the group
$G$ has property (T).
\end{proposition}
\begin{proof}
Suppose first that $G$ has property (T). Let $c : G \recht \cU(M \ovt \B(K))$ be
a $1$-cocycle of $G \actson M$ having $\xi_n \in M \ovt K$ as a sequence of almost
invariant unit vectors. Define the unitaries $u_g$ on $\rL^2(M,\tau)$ by extending $\al_g$. Then,
$$
\pi : G \recht \cU(\rL^2(M,\tau) \ot K) : \pi(g) = c(g) (u_g \ot 1)
$$
is a unitary representation of $G$ and we can view $\xi_n$ as a sequence of
almost invariant unit vectors. Since $G$ has property (T), $\pi$ admits a
unit invariant vector. Even more, we find a sequence $\eta_n \in \rL^2(M,\tau) \ot K$
of $\pi$-invariant vectors satisfying $\eta_n - \xi_n \recht 0$.

Since $\|\xi_n\| = 1$, it follows that $\|\xi_n^* \eta_n - 1 \|_2 \recht 0$. Hence, the
right support projection of $\xi_n^* \eta_n$ converges strongly to $1$. A fortiori,
the right support projection of $\eta_n$ converges to $1$. We view $\eta_n$ as a
closed operator from $\rL^2(M,\tau)$ to $\rL^2(M,\tau) \ot K$. Taking the polar
decomposition of $\eta_n$, we find $v_n \in M \ovt K$ satisfying $c(g) (\al_g \ot \id)(v_n) =
v_n$ for all $g \in G, n \in \N$ and such that $v_n^* v_n$ is a sequence of projections in $M$
converging strongly to $1$. Define the von Neumann algebra
$$N = \begin{pmatrix} M \ovt \B(K) & M \ovt K \\ (M \ovt K)^* & M \end{pmatrix} \; .$$
Define the action $(\gamma_g)$ of $G$ on $N$ by
$$\gamma_g\begin{pmatrix} a & b \\ e & f \end{pmatrix} = \begin{pmatrix}
c(g) (\al_g \ot \id)(a) c(g)^* & c(g) (\al_g \ot \id)(b) \\ (\al_g \ot \id)(e)c(g)^* &
\al_g(f) \end{pmatrix} \; .$$
Define $p = \bigl(\begin{smallmatrix} 1 & 0 \\ 0 & 0 \end{smallmatrix}\bigr)$,
$q = 1-p$ and $w_n = \bigl(\begin{smallmatrix} 0 & v_n \\ 0 & 0 \end{smallmatrix}\bigr)$.
Then, $w_n$ is a sequence of partial isometries in the fixed point algebra $N^G$,
satisfying $w_n \in p N^G q$ and $w_n^* w_n \recht q$ strongly. It follows that
$q \prec p$ in the von Neumann algebra $N^G$. So, we find $v \in M \ovt K$
satisfying $c(g) (\al_g \ot \id)(v) = v$ for all $g$ and $v^* v = 1$.
Hence, $G \actson M$ has property (T).

Suppose conversely that $G \actson M$ has property (T). Let
$\pi : G \recht \cU(K)$ be a strongly continuous unitary
representation of $G$ admitting $\xi_n$ as a sequence of
almost invariant unit vectors. By \cite[Theorem 2.12.9]{BHV},
it is sufficient to prove that $\pi$ has a non-zero finite dimensional
$\pi(G)$-invariant subspace. Define $c : G \recht \cU(M \ovt \B(K)) : c(g) = 1 \ot \pi(g)$.
Obviously, $c$ is a $1$-cocycle of $G \actson M$ having $1 \ot \xi_n$ as a sequence of almost
invariant unit vectors. By property (T) of $G \actson M$, we find $\xi \in M \ovt K$
satisfying $\xi^* \xi = 1$ and $c(g)(\al_g \ot \id)(\xi) = \xi$ for all $g \in G$.
Denoting again by $u : g \mapsto u_g$ the representation of $G$ on $\rL^2(M,\tau)$
obtained by extending $\al_g$, we find that $u \ot \pi$ admits an invariant unit
vector $\xi$. Identify $\rL^2(M,\tau) \ot K$ with the Hilbert space of Hilbert-Schmidt
operators from $\overline{\rL^2(M,\tau)}$ to $K$. Then, $T:=\xi \xi^*$ is a non-zero
trace-class operator on $K$ satisfying $\pi(g) T \pi(g)^* = T$ for all $g \in G$. So,
for $\eps > 0$ sufficiently small, the spectral projection $\chi_{[\eps,+\infty)}(T)$
projects onto a non-zero finite dimensional $\pi(G)$-invariant subspace of $K$.
\end{proof}

The following is a slight generalization of \cite[Theorem 5.3]{claire}. When
$H \actson M$ is an action, we denote by $M^H$ the von Neumann algebra of $H$-fixed points.

\begin{lemma} \label{lemma.induced}
Let $G \actson[\al] M$ be an action of the l.c.s.c.\ group $G$ on the von Neumann
algebra $M$. Let $H \lhd G$ be a closed normal subgroup and assume that there is a
$^*$-isomorphism $\theta : M \recht \rL^\infty(H) \ovt M^H$ satisfying
$\theta \circ \al_h = (\rho_h \ot \id)\circ \theta$ for all $h \in H$,
where $\rho_h$ denotes the right translation by $h$ on $\rL^\infty(H)$.

Then, $G \actson M$ has property (T) if and only if $G/H \actson M^H$ has property (T).
\end{lemma}
\begin{proof}
We say that two $1$-cocycles $c_1,c_2$ of $G \actson[\al] M$ with values in $\cU(K)$
are unitarily equivalent if there exists a unitary $v \in \cU(M \ovt \B(K))$
satisfying $c_1(g) = v c_2(g) (\al_g \ot \id)(v^*)$ for all $g \in G$.
We denote by $\rH^1(G \actson[\al] M,\cU(K))$ the set of equivalence classes of $1$-cocycles.

In the first part of the proof, we show that the obvious map $$\Theta :
\rH^1(G/H \actson M^H, \cU(K)) \recht \rH^1(G \actson M,\cU(K)) : \Theta(c) =
c \circ \pi \quad\text{with}\;\; \pi : G \recht G/H \; ,$$ is a bijection.
In the second part of the proof, we show that this map and its inverse preserve
the property of having invariant, resp.\ almost invariant, vectors. Both parts
together show that property (T) of $G \actson M$ is equivalent with property (T)
of $G/H \actson M^H$.

It is straightforward to check that $\Theta$ is well defined and injective.
Suppose that $c : G \recht \cU(M \ovt \B(K))$ is a $1$-cocycle of $G \actson M$.
In order to prove that $c$ is in the range of $\Theta$, it suffices to prove
that $c$ is unitarily equivalent with $c'$ satisfying $c'(h) = 1$ for all $h \in H$.
Identify throughout $\cU(M \ovt \B(K))$ with $\cU(\rL^\infty(H) \ovt M^H \ovt \B(K))$
and view the latter as measurable functions $H \recht \cU(M^H \ovt \B(K))$,
modulo equality almost everywhere. By \cite[Theorem 1]{Moore3},
take a measurable map $\vphi : H \times H \recht \cU(M^H \ovt \B(K))$
such that $c(h) = \vphi( \cdot , h)$ for all $h \in H$. Since $c$ is a
cocycle, we find that
$$\vphi(k,hg) = \vphi(k,h) \vphi(kh,g) \quad\text{for almost all}\;\;
(k,h,g) \in H \times H \times H \; .$$
By the Fubini theorem, take $k_0 \in H$ such that for almost all $(h,g)
\in H \times H$, the previous equality holds for $(k_0,h,g)$. Define the
unitary $v \in \cU(M \ovt \B(K))$ as $v = \vphi(k_0,k_0^{-1} \cdot)^*$
and set $c'(g) = v c(g) (\al_g \ot \id)(v^*)$. By construction, $c'(g) = 1$
for almost all $g \in H$ and hence for all $g \in H$ by continuity.

It is an exercise to check that the $1$-cocycle $c \in \rH^1(G/H \actson M^H,\cU(K))$
has a unit invariant vector if and only if $\Theta(c)$ has. Also, a sequence of
almost invariant unit vectors for $c$ defines a sequence of almost invariant
unit vectors for $\Theta(c)$. Finally, suppose that $\xi_n \in M \ovt K$ is a
sequence of almost invariant unit vectors for $\Theta(c)$. In order to conclude
the proof of the lemma, it suffices to show that there exists a sequence
$\eta_n \in M^H \ovt K$ satisfying $\eta_n^* \eta_n = 1$ for all $n$ and
$\xi_n - \eta_n \recht 0$ $^*$-strongly.

Define the sets $X = \{v \in M \ovt K \mid v^* v = 1\}$ and $Y = \{v \in M^H
\ovt K \mid v^* v = 1 \}$. Through the isomorphism $M \cong \rL^\infty(H) \ovt M^H$,
we identify $X$ with the set of measurable functions from $H$ to $Y$
(modulo equality almost everywhere). Take a bounded metric $d_0$ on $Y$
inducing the strong$^*$ topology. Let $\mu$ be a probability measure on $H$
in the same measure class as the Haar measure. Following \cite[page 5]{Moore3},
define the metric $d$ on $X$ by
$$d(v,w) = \int_H d_0(v(h),w(h)) d\mu(h) \; .$$
Then, $d$ induces the strong$^*$ topology on $X$. It is easy to check that
when $v_n,w_n \in X$ such that $d(v_n,w_n)$ is summable, then $v_n(h) - w_n(h)
\recht 0$ $^*$-strongly for almost every $h \in H$ (see \cite[Proposition 6]{Moore3}).

View $Y \subset X$ as constant functions. We have to prove that $d(\xi_n,Y) \recht 0$.
Suppose the contrary. Write $H$ as an increasing union of compact subsets $H_n$.
After passage to a subsequence, we find $\eps > 0$ such that $d(\xi_n,Y) > \eps$
for all $n$ and such that
$$d(\xi_n , (\al_g \ot \id)(\xi_n)) < 2^{-n} \quad\text{for all}\;\; n \in \N, g \in H_n \; .$$
It follows that for all $g \in H$, we have
$$d_0(\xi_n(h), \xi_n(hg)) \recht 0 \quad\text{for almost all}\;\; h \in H \; .$$
By the Fubini theorem, take $h_0 \in H$ such that $d_0(\xi_n(h_0), \xi_n(h_0 g))
\recht 0$ for almost all $g \in H$. It follows that $d(\xi_n,\xi_n(h_0)) \recht 0$,
contradicting the assumption that $d(\xi_n,Y) > \eps$ for all $n$.
\end{proof}

\begin{proposition}\label{prop.leftright}
Let $\cG$ be a l.c.s.c.\ group with closed subgroups $H_1,H_2$. Then,
$H_1 \actson \rL^\infty(\cG/H_2)$ has property (T) if and only if $H_2
\actson \rL^\infty(\cG/H_1)$ has property (T).
\end{proposition}
\begin{proof}
Set $M = \rL^\infty(\cG)$ and $G = H_1 \times H_2$ acting by left-right
translations on $M$~:
$$(\al_{(g,h)}(F))(x) = F(g^{-1} x h) \quad \text{for all}\quad F \in \rL^\infty(\cG),
x \in \cG, g \in H_1, h \in H_2 \; .$$
We apply Lemma \ref{lemma.induced} to $G \actson M$ and the closed normal subgroups
$H = H_i$, $i=1,2$ of $G$. By the Effros-Mackey theorem
(see e.g.\ \cite[Theorem II.12.17]{kechris}),
the quotient map $\cG \recht \cG/H$ admits a Borel
lifting and hence, there exists an $H$-equivariant
isomorphism $M \recht \rL^\infty(H) \ovt M^H$. So,
by Lemma \ref{lemma.induced}, property (T) of $G \actson M$
is equivalent with property (T) of $H_1 \actson \rL^\infty(\cG/H_2)$
as well as with property (T) of $H_2 \actson \rL^\infty(\cG/H_1)$.
\end{proof}

\section{The lattice actions $\Gamma \actson \R^n$ have property (T)}

Recall from \cite{Fu} that if $\Gamma \actson (X,\mu)$ is a free
ergodic p.m.p.\ action, then property (T) of $\cR(\Gamma \actson X)$
in the sense of Zimmer, is equivalent with property (T) of the group
$\Gamma$ (see also \cite[Proposition 2.4]{Z1} and Proposition
\ref{prop.Tpmp} above).

If $\Gamma$ is a property (T) group, the fundamental group of $\cR(\Gamma \actson X)$
is countable for any free ergodic p.m.p.\ action (see \cite[Corollary 1.8]{gg} if
$\Gamma$ is moreover ICC and see \cite[Theorem 5.9]{ioana} for the general case).

But more is true: we proved in \cite[Theorem 6.1]{proc-Connes} that the fundamental group of a II$_1$ equivalence relation $\cR$ on $(X,\mu)$ is countable whenever the full group $[\cR]$ contains a property (T) group that implements an ergodic action on $(X,\mu)$. As a result, the following theorem is rather surprising: we obtain a II$_1$ equivalence relation $\cR$ with property (T) and fundamental group $\R_+$; hence, none of the $\cR^t$ can be implemented by a free action of a group and none of the $\cR^t$ can be implemented by a possibly non-free action of a property (T) group.

\begin{theorem} \label{thm.equiv-T}
Let $\Gamma < \SL(n,\R)$ be a lattice and let $\cR$ be the II$_1$ equivalence relation obtained by restricting the orbit
equivalence relation $\cR(\Gamma \actson \R^n)$ to a set of Lebesgue measure
$1$. If $n \geq 4$, the equivalence relation $\cR$ has property (T) in the sense
of Zimmer, but nevertheless $\cF(\cR) = \Rp$. In particular,
\begin{itemize}
\item none of the equivalence relations $\cR^t$, $t > 0$, can be implemented by a free action of a group,
\item none of the equivalence relations $\cR^t$, $t > 0$, can be implemented by a possibly non-free action of a property (T) group.
\end{itemize}
\end{theorem}

\begin{proof}
Proving property (T) of $\cR$ amounts to proving property (T) for the
action $\Gamma \actson \rL^\infty(\R^n)$.

Define the l.c.s.c.\ group $\cG = \SL(n,\R)$ and set $H_1 = \Gamma$.
Consider the linear action $\cG \actson \R^n$ and set $H_2 = \{A \in \SL(n,\R)
\mid A e_1 = e_1\}$, where $e_1$ denotes the first basis vector of $\R^n$.
By construction, the action $\Gamma \actson \rL^\infty(\R^n)$ can be
viewed as $H_1 \actson \rL^\infty(\cG/H_2)$. Hence, by Proposition \ref{prop.leftright},
property (T) for this last action is equivalent with property (T) of $H_2 \actson
\rL^\infty(\cG/H_1)$. This action admits a finite invariant measure, because $H_1$
is a lattice in $\cG$. Moreover, $H_2 \cong \SL(n-1,\R) \ltimes \R^{n-1}$, which
has property (T) for $n \geq 4$. So, it follows from Proposition \ref{prop.Tpmp}
that $H_2 \actson \rL^\infty(\cG/H_1)$ has property (T).

The action on $\R^n$ by multiples of the identity matrix scales the Lebesgue
measure and commutes with the action of $\Gamma$. Hence, the fundamental
group of $\cR$ equals $\R_+$.
The statements about implementing $\cR^t$ by group actions, follow from the discussion preceding the theorem.
\end{proof}

Note that in the case $n=2$, a similar reasoning yields the following result
of \cite{aubert}: the action $\SL(2,\Z) \actson \R^2$ is amenable and hence,
$\rL^\infty(\R^2) \rtimes \SL(2,\Z)$ is isomorphic with the unique hyperfinite II$_\infty$ factor.

\section{Cocycle and OE superrigidity theorems}

\subsection{Proof of Theorem \ref{thm.cocycle-superrigid-concrete}}

We prove in this section the cocycle superrigidity theorem \ref{thm.cocycle-superrigid-concrete} as a consequence of the more general Theorem \ref{thm.cocyclesuperrigid} below.

We do not know whether $\SL(n,\R) \actson \R^n$ is $\Ufin$-cocycle superrigid for $n=3,4$. On the other hand, some condition is needed on the Polish group in which the $1$-cocycle takes its values. Indeed, almost by construction, we have the following result, that we also prove at the end of this subsection.

\begin{proposition}\label{prop.notuntwist}
Let $n \geq 3$. The action $\SL(n,\Z) \actson \R^n$ admits a $1$-cocycle with values in \linebreak $\SL(n-1,\R) \ltimes \R^{n-1}$ that is not cohomologous to a group morphism.
\end{proposition}

Recall from \cite{Pcocycle} the following definition of $s$-malleability of a measure preserving action.

\begin{definition}
Let $\Gamma$ be a locally compact second countable (l.c.s.c.) group and $\Gamma \actson (X,\mu)$ an action preserving the finite or infinite measure $\mu$. The action is called \emph{$s$-malleable} if there exists
\begin{itemize}
\item a one-parameter group $(\al_t)_{t \in \R}$ of measure preserving transformations of $X \times X$,
\item an involutive measure preserving transformation $\be$ of $X \times X$,
\end{itemize}
such that
\begin{itemize}
\item $\al_t$ and $\be$ commute with the diagonal action $\Gamma \actson X \times X$,
\item $\al_1(x,y) \in \{y\} \times X$ for almost all $(x,y) \in X \times X$,
\item $\beta(x,y) \in \{x\} \times X$ for almost all $(x,y) \in X \times X$,
\item $\al_t \circ \beta = \beta \circ \al_{-t}$ for all $t \in \R$.
\end{itemize}
\end{definition}

Theorem 0.1 in \cite{Pcocycle} says the following. Let $\Gamma \actson (X,\mu)$ be an $s$-malleable, probability measure preserving action and $\Lambda < \Gamma$ a normal subgroup with the relative property (T) such that the restriction of $\Gamma \actson (X,\mu)$ to $\Lambda$ is weakly mixing. Then every $1$-cocycle of $\Gamma \actson (X,\mu)$ with values in a Polish group of finite type, is cohomologous to a group morphism.

Recall here that one of the equivalent formulations of weak mixing for a p.m.p.\ action $\Lambda \actson (X,\mu)$ is the ergodicity of the diagonal action $\Lambda \actson (X \times X,\mu \times \mu)$. If $\Lambda \actson (X,\mu)$ is a weakly mixing p.m.p.\ action and $\Gamma \actson (Y,\eta)$ is any ergodic p.m.p.\ action, then the diagonal action $\Gamma \actson X \times Y$ is ergodic. In particular, the diagonal action
$$\Gamma \actson \underbrace{X \times \cdots \times X}_{k \; \text{times}}$$
is ergodic for every $k$, once it is ergodic for $k = 2$. For infinite measure preserving actions, things are more complicated and, for instance, the diagonal action
$$\SL(n,\Z) \actson \underbrace{\R^n \times \cdots \times \R^n}_{k \; \text{times}}$$
is ergodic if and only if $k \leq n-1$. This partially explains the formulation of the following result.

\begin{theorem} \label{thm.cocyclesuperrigid}
Let $\Gamma \actson (X,\mu)$ be an infinite measure preserving, $s$-malleable action. Assume that
\begin{itemize}
\item the diagonal action $\Gamma \actson X \times X$ has property (T),
\item the $4$-fold diagonal action $\Gamma \actson X \times X \times X \times X$ is ergodic.
\end{itemize}
Then, $\Gamma \actson (X,\mu)$ is $\Ufin$-cocycle superrigid.
\end{theorem}

The proof of Theorem \ref{thm.cocyclesuperrigid} follows entirely the setup of the proof of \cite[Theorem 0.1]{Pcocycle}. But, one has to be careful at those places in \cite{Pcocycle} where weak mixing is applied. To proper way to deal with these issues, lies in the following lemma distilled from the proof of \cite[Lemmas 3.1]{furman-on-popa}.

\begin{lemma} \label{lemma.tata}
Let $(Z,d)$ be a Polish space with separable complete metric $d$ and $(\al_g)_{g \in \cG}$ a continuous action of a Polish group $\cG$ by homeomorphisms of $Z$. Assume that $d$ is $(\al_g)_{g \in \cG}$-invariant.

Let $\Gamma$ be a l.c.s.c.\ group and $\Gamma \actson (X,\mu)$, $\Gamma \actson (Y,\eta)$ non-singular actions. Let $F : X \times Y \recht Z$ be a measurable map satisfying
$$F(g \cdot x,g \cdot y) = \al_{\om(g,x)}(F(x,y))$$
for almost all $(x,y) \in X \times Y$, $g \in \Gamma$, where $\om : \Gamma \times X \recht \cG$ is some measurable map.

Consider the diagonal action $\Gamma \actson X \times Y \times Y$ and assume that $\rL^\infty(X \times Y \times Y)^\Gamma = \rL^\infty(X)^\Gamma \ot 1 \ot 1$ (which holds in particular if the diagonal action $\Gamma \actson X \times Y \times Y$ is ergodic). Then, there exists a measurable $H : X \recht Z$ with $F(x,y) = H(x)$ for almost all $(x,y) \in X \times Y$.
\end{lemma}
\begin{proof}
Define the map $G : X \times Y \times Y \recht \R : G(x,y,z) = d(F(x,y),F(x,z))$. Since $d$ is $(\al_g)_{g \in \cG}$-invariant, the map $G$ is invariant under the diagonal $\Gamma$-action. By our assumption, $G(x,y,z) = G_0(x)$ for almost all $(x,y,z)$ and some measurable map $G_0 : X \recht \R_+$. We claim that $G_0(x) = 0$ for almost all $x \in X$. Let $\delta > 0$ and assume that $G_0(x) \geq \delta$ for all $x$ in a non-negligible subset $\cU$ of $X$. Cover $Z$ by a sequence $(B_n)_{n \in \N}$ of balls of diameter strictly smaller than $\delta$. Write, for every $x \in X$, $F_x : y \mapsto F(x,y)$. By the Fubini theorem, for almost every $x \in \cU$, we have $G(x,y,z) \geq \delta$ for almost all $(y,z) \in Y \times Y$. Hence, for almost every $x \in \cU$, we have that $F_x^{-1}(B_n)$ is negligible for every $n$, which is absurd.

So, $G(x,y,z) = 0$ for almost all $(x,y,z) \in X \times Y \times Y$. Again by the Fubini theorem, take $z \in Y$ such that $d(F(x,y),F(x,z)) = 0$ for almost all $(x,y) \in X \times Y$. Putting $H(x) := F(x,z)$, we are done.
\end{proof}

\begin{proof}[Proof of Theorem \ref{thm.cocyclesuperrigid}]
Let $N$ be a II$_1$ factor and $\cG \subset \cU(N)$ a closed subgroup. Let $\om : \Gamma \times X \recht \cG$ be a $1$-cocycle, meaning that for all $g,h \in \Gamma$, we have
$$\om(gh,x) = \om(g, h \cdot x) \, \om(h,x) \quad\text{for almost all}\;\; x \in X \; .$$
Define the following $1$-cocycles for the diagonal action $\Gamma \actson X \times X$.
$$\om_0 : \Gamma \times X \times X \recht \cG : \om_0(g,x,y) = \om(g,x) \quad\text{and}\quad
\om_t : \Gamma \times X \times X \recht \cG : \om_t(g,x,y) = \om_0(g,\al_t(x,y)) \; .$$

Define the action $(\rho_g)_{g \in \Gamma}$ of $\Gamma$ by automorphisms of $\rL^\infty(X) \ovt N$ by the formula
$$(\rho_{g^{-1}}(F))(x) = \om(g,x)^* F(g \cdot x) \om(g,x) \; .$$
Denote by $B$ the von Neumann subalgebra of $(\rho_g)_{g \in \Gamma}$-fixed points.

{\bf Claim.} Whenever $p$ is a non-zero projection in $B$, there exists a measurable function $\vphi : X \times X \recht N$ and a non-zero projection $q \in B$ such that $q \leq p$ and such that for all $g \in \Gamma$, we have
$$\om(g,x) \vphi(x,y) = \vphi(g \cdot x , g \cdot y) \om(g,y) \quad , \quad \vphi(x,y) \vphi(x,y)^* = q(x) \quad , \quad \vphi(x,y)^* \vphi(x,y) = q(y)$$
for almost all $(x,y) \in X \times X$.

{\bf Proof of the claim.} Since $p$ is $(\rho_g)_{g \in \Gamma}$-invariant, the function $x \mapsto \tau(p(x))$ is $\Gamma$-invariant and hence constantly equal to $0 < \lambda \leq 1$. Let $p_0 \in P$ be a projection with $\tau(p_0) = \lambda$. It follows that, inside $\rL^\infty(X) \ovt N$, the projections $p$ and $1 \ot p_0$ are equivalent. Take a partial isometry $v \in \rL^\infty(X) \ovt N$ such that $v^* v = p$ and $vv^* = 1 \ot p_0$. Define $\eta(g,x) = v(g \cdot x) \om(g,x) v(x)^*$ and note that $\eta$ is a $1$-cocycle for $\Gamma \actson X$ with values in $\cU(p_0 N p_0)$. Set $\eta_0(g,x,y) = \eta(g,x)$ and, for all $n \geq 1$, $\eta_n(g,x,y) = \eta_0(g,\al_{2^{-n}}(x,y))$.

Define the Hilbert space $K = \bigoplus_{k=1}^\infty p_0 \rL^2(N) p_0$. We define the following $1$-cocycle of $\Gamma \actson X \times X$ with values in the unitary group $\cU(K)$ of $K$.
$$(c(g,x,y) \xi)_k = \eta(g,x) \, \xi_k \, \eta_k(g,x,y)^* \; .$$
Define the map $\xi_n : X \recht (K)_1$ by the formula
$$(\xi_n(x))_k = \begin{cases} \tau(p_0)^{-1/2} p_0 &\quad\text{if}\;\; k = n \; , \\ 0 &\quad\text{if}\;\; k \neq n \; .\end{cases}$$
One checks that $\xi_n$ is a sequence of almost invariant unit vectors. Since $\Gamma \actson X \times X$ has property (T), we find a unit invariant vector, i.e.\ a measurable map $\xi : X \times X \recht K$ with $\|\xi(x,y)\| = 1$ for almost all $(x,y)$ and, for all $g \in \Gamma$,
$$\xi(g \cdot x, g \cdot y) = c(g,x,y)  \, \xi(x,y)$$
almost everywhere. It follows that
$$\xi(g \cdot x,g \cdot y)_k = \eta(g,x)  \, \xi(x,y)_k  \, \eta_k(g,x,y)^*$$
almost everywhere. In particular, for every $k$, the function $(x,y) \mapsto \|\xi(x,y)_k\|$ is $\Gamma$-invariant and hence, constant. Since $\|\xi(x,y)\|=1$ for almost all $(x,y)$, we can pick $k$ such that $\|\xi(x,y)_k\| \neq 0$ for almost all $(x,y)$. Taking the polar decomposition of $\xi(x,y)_k$, we find a non-zero partial isometry $\psi \in \rL^\infty(X \times X) \ovt p_0 N p_0$ satisfying
$$\psi(g \cdot x,g \cdot y)  \, \eta_k(g,x,y) = \eta(g,x)  \, \psi(x,y)$$
almost everywhere. We set $\vphi_0(x,y) := v(x)^* \psi(x,y) v_k(x,y)$, where $v_0(x,y) = v(x)$ and $v_k(x,y) = v(\al_{2^{-k}}(x,y))$. It follows that
$$\om(g,x) \,  \vphi_0(x,y) = \vphi_0(g \cdot x,g \cdot y)  \, \om_{t_0}(g,x,y)$$
where $t_0 = 2^{-k}$.

Set $r(x,y) = \vphi_0(x,y) \vphi_0(x,y)^*$. It follows that
$$r(g \cdot x,g \cdot y) = \om(g,x)  \, r(x,y)  \, \om(g,x)^*$$
almost everywhere. By Lemma \ref{lemma.tata}, we find a projection $q \in \rL^\infty(X) \ovt N$ such that $r(x,y) = q(x)$ almost everywhere. Then, $q$ is a non-zero projection in $B$ and $q \leq p$. Also, $\vphi_0(x,y) \vphi_0(x,y)^* = q(x)$ almost everywhere.

Set $q_0(x,y) = q(x)$ and $q_t(x,y) = q(\al_t(x,y))$.
We now construct $\vphi_1 : X \times X \recht N$ such that $\vphi_1(x,y) \vphi_1(x,y)^* = q(x)$, $\vphi_1(x,y)^* \vphi_1(x,y) = q_{2t_0}(x,y)$ and
$$\om(g,x) \,  \vphi_1(x,y) = \vphi_1(g \cdot x, g \cdot y)  \, \om_{2t_0}(g,x,y)$$
almost everywhere. Continuing the same procedure $k$ times (remember that $t_0 = 2^{-k}$), we will have found $\vphi = \vphi_k : X \times X \recht N$ satisfying $\vphi(x,y) \vphi(x,y)^* = q(x)$, $\vphi(x,y)^* \vphi(x,y) = q_1(x,y) = q(y)$ and
$$\om(g,x)  \, \vphi(x,y) = \vphi(g \cdot x , g \cdot y)  \, \om_1(g,x,y) = \vphi(g \cdot x , g \cdot y)  \, \om(g,y) \; ,$$
hence proving the claim. In fact, it suffices to take $$\vphi_1(x,y) = \vphi_0(x,y) \vphi_0(\beta(\al_{2t_0}(x,y)))^*$$ and to use that $(\al_t)_{t \in \R}$ is a one-parameter group, $\beta \circ \al_t = \al_{-t} \circ \beta$ and $\beta(x,y) \in \{x\} \times Y$ for almost all $(x,y)$. So, the claim above has been proven.

Using the claim and a maximality argument, we find a measurable function $\vphi : X \times X \recht \cU(N)$ such that
$$\om(g,x)  \, \vphi(x,y) = \vphi(g \cdot x,g \cdot y)  \, \om(g,y)$$
almost everywhere. Set $H(x,y,z) = \vphi(x,y) \vphi(y,z)$. It follows that
$$H(g \cdot x,g \cdot y,g \cdot z) = \om(g,x)  \, H(x,y,z)  \, \om(g,z)^*$$
for almost all $(x,y,z)$. By Lemma \ref{lemma.tata} and because the $4$-fold diagonal action $\Gamma \actson X \times X \times X \times X$ is ergodic, $H$ is essentially independent of its second variable. So, we find a measurable $F : X \times X \recht \cU(N)$ such that $\vphi(x,y) = F(x,z) \vphi(y,z)^*$ for almost every $(x,y,z)$. By the Fubini theorem, take $z \in X$ such that the previous formula holds for almost all $(x,y)$. Set $\psi(x) = F(x,z)^*$ and $G(y) = \vphi(y,z)^*$. It follows that
$$\psi(g \cdot x)  \, \om(g,x) \,  \psi(x)^* = G(g \cdot y)  \, \om(g,y)  \, G(y)^*$$
almost everywhere. Hence, the left-hand side is independent of $x$ and we have found a group morphism $\delta : G \recht \cU(N)$ and a measurable map $\psi : X \recht \cU(N)$ such that
$$\om(g,x) = \psi(g \cdot x)^* \,  \delta(g) \,  \psi(x)$$
almost everywhere.

Consider the quotient Polish space $\cU(N)/\cG$ with the induced metric, which is invariant under left multiplication by elements of $\cU(N)$.
Write $F : X \recht \cU(N)/\cG : F(x) = \psi(x) \cG$. It follows that $F(g \cdot x) = \delta(g) F(x)$ almost everywhere. By Lemma \ref{lemma.tata}, $F$ is essentially constant. So, we find a unitary $u \in \cU(N)$ and a measurable map $w : X \recht \cG$ such that $\psi(x) = u w(x)$ almost everywhere. Replacing $\delta$ by $u^* \delta(\cdot) u$, it follows that $\om(g,x) = w(g \cdot x)^* \delta(g) w(x)$ almost everywhere. In particular, $\delta(g) \in \cG$ and we are done.
\end{proof}

As a principle (cf.\ \cite[Proposition 3.6]{Pcocycle}), once the restriction of a $1$-cocycle $\om : G \times X \recht \cG$ to a closed subgroup $H < G$, is cohomologous to a group morphism $H \recht \cG$ and if $H$ is sufficiently normal in $G$ and acts sufficiently mixingly on $X$, the entire $1$-cocycle $\om$ is cohomologous to a group morphism $G \recht \cG$. In our setting, we need the following.

\begin{lemma}\label{lemma.extend}
Let $G \actson (X,\mu)$ be a non-singular action of the l.c.s.c.\ group $G$ and $\om : G \times X \recht \cG$ a $1$-cocycle with values in the Polish group of finite type $\cG$. Let $H < G$ be a closed subgroup and assume that $\om|_H$ is cohomologous to a group morphism $H \recht \cG$. If for every $g \in G$, the diagonal action of the group $H \cap g H g^{-1}$ on $X \times X$ is ergodic, then $\om$ is cohomologous to a morphism $G \recht \cG$.
\end{lemma}
\begin{proof}
We may assume that for every $h \in H$, we have $\om(h,x) = \delta(h)$ for almost every $x \in X$, where $\delta : H \recht \cG$ is a continuous group morphism. Let $g \in G$ and put $F(x) = \om(g,x)$. Using the cocycle equation, it follows that for all $h \in H \cap g^{-1} H g$, we have $F(h \cdot x) = \delta(ghg^{-1}) F(x) \delta(h)^{-1}$ almost everywhere. By Lemma \ref{lemma.tata}, $F$ is essentially constant. So, we have shown that for every $g \in G$, the map $x \mapsto \om(g,x)$ is essentially constant. It follows that $\om(g,x) = \delta(g)$.
\end{proof}

After showing the following lemma, we can prove Theorem \ref{thm.cocycle-superrigid-concrete} and Proposition \ref{prop.notuntwist}.

\begin{lemma} \label{lemma.ergodic-T}
Let $\Gamma$ be a lattice in $\SL(n,\R)$ and consider the linear action of $\Gamma$ on $\R^n$. Let
$$\Gamma \actson X^{(k)} := \underbrace{\R^n \times \cdots \times \R^n}_{\text{$k$ times}}$$
be the $k$-fold diagonal action. Then, $\Gamma \actson X^{(k)}$
\begin{itemize}
\item is ergodic if and only if $k \leq n-1$,
\item has property (T) if and only if $k \leq n-3$ or $k \geq n$ (the latter part being as interesting as the trivial group $\{e\}$ having property (T)).
\end{itemize}
\end{lemma}
\begin{proof}
Writing the elements of $\R^n$ as column vectors, identify, up to measure zero, $X^{(n)}$ with $\GL(n,\R)$, with the $\Gamma$-action given by left multiplication. The determinant function is invariant and not essentially constant, proving that $\Gamma \actson X^{(k)}$ is non-ergodic for $k \geq n$. It also follows that for $k \geq n$, the action $\Gamma \actson X^{(k)}$ is essentially free and proper. Hence, it has property (T) because the trivial group has property (T).

Let now $k \leq n-1$. Denoting by $(e_i)_{i=1,\ldots,n}$ the standard basis vectors in $\R^n$, the orbit of $(e_1,\ldots,e_k) \in X^{(k)}$ under the diagonal $\SL(n,\R)$-action has complement of measure zero, so that we can identify $\Gamma \actson X^{(k)}$ with $\Gamma \actson \SL(n,\R)/H$, where $H = \{A \in \SL(n,\R) \mid Ae_i = e_i \;\;\text{for all}\;\; i=1,\ldots,k\}$. Observe that
$H \cong \SL(n-k,\R) \ltimes \M_{n-k}(\R)$, where $\M_{n-k}(\R)$ denotes the additive group of $(n-k)\times(n-k)$ matrices on which $\SL(n-k,\R)$ acts by multiplication. By Moore's ergodicity theorem (see e.g.\ \cite[Theorem 2.2.6]{Z2}), $H \actson \Gamma \backslash \SL(n,\R)$ is ergodic and hence, $\Gamma \actson \SL(n,\R)/H$ is ergodic.

By Proposition \ref{prop.leftright}, property (T) of $\Gamma \actson \SL(n,\R) / H$ is equivalent with property (T) of $H \actson
\Gamma \backslash \SL(n,\R)$. By Proposition \ref{prop.Tpmp}, the latter is equivalent with the group $H$ having property (T), which is in turn equivalent with $n-k \geq 3$.
\end{proof}

\begin{proof}[Proof of Theorem \ref{thm.cocycle-superrigid-concrete}]
Observe that part 1 is a special case of part 3, by taking $k=1$ and $H = \{1\}$. We start by proving part 3. Set $X = \M_{n,k}(\R)$. The action $\Gamma \actson X$ by left multiplication is $s$-malleable. It suffices to take
$$\al_t(A,B) = (\cos(\pi t/2) A + \sin(\pi t/2)A,-\sin(\pi t/2) A + \cos (\pi t/2) B) \quad\text{and}\quad \beta(A,B) = (A,-B)$$
whenever $t \in \R$ and $A,B \in \M_{n,k}(\R)$.

Moreover, $\Gamma \actson X$ can be viewed as the $k$-fold diagonal action $\Gamma \actson \R^n \times \cdots \times \R^n$.
By Lemma \ref{lemma.ergodic-T} and because $n \geq 4k+1$, the diagonal action $\Gamma \actson X \times X$ has property (T) and the $4$-fold diagonal action $\Gamma \actson X \times X \times X \times X$ is ergodic. So, by Theorem \ref{thm.cocyclesuperrigid}, $\Gamma \actson X$ is $\Ufin$-cocycle superrigid. Since the diagonal action $\Gamma \actson X \times X$ is ergodic and $\Gamma$ is a normal subgroup of $G$, Lemma \ref{lemma.extend} implies that $G \actson X$ is $\Ufin$-cocycle superrigid.

It remains to prove part 2 of the theorem. By part 1, we already know $\Gamma \actson \R^n$ is $\Ufin$-cocycle superrigid. Define, for every $x \in \Z^n$, $\Gamma_x = \{g \in \Gamma \mid g x = x \}$. We claim that the diagonal action $\Gamma_x \actson \R^n \times \R^n$ is ergodic for every $x \in \Z^n$. Once this claim is proven, Lemma \ref{lemma.extend} implies that $\Gamma \ltimes \Z^n \actson \R^n$ is $\Ufin$-cocycle superrigid.

For $x = 0$, the claim follows from Lemma \ref{lemma.ergodic-T}. Let now $x \neq 0$. Define the closed subgroup $H = \{g \in \SL(n,\R) \mid g e_1 = e_1 \}$ of $\SL(n,\R)$. Exactly as in the proof of Lemma \ref{lemma.ergodic-T}, when $n \geq 4$, the diagonal action $\Lambda \actson \R^n \times \R^n$ of any lattice $\Lambda \subset H$ is ergodic. Take $g_0 \in \SL(n,\Q)$ with $x = g_0 e_1$. Since $\Gamma$ is a finite index subgroup of $\SL(n,\Z)$, it follows that $g_0^{-1} \Gamma_x g_0$ contains a finite index subgroup of $\{g \in \SL(n,\Z) \mid g e_1 = e_1 \}$ and hence, is a lattice in $H$. So, its diagonal action on $\R^n \times \R^n$ is ergodic. Then, the same is true for the diagonal action of $\Gamma_x$ on $\R^n \times \R^n$.
\end{proof}

\begin{proof}[Proof of Proposition \ref{prop.notuntwist}]
Write $G = \SL(n,\R)$ and $\Gamma = \SL(n,\Z)$. As before, identify $\Gamma \actson \R^n$ with $\Gamma \actson G / H$, where $H \cong \SL(n-1,\R) \ltimes \R^{n-1}$. Let $\theta : G / H \recht G$ be a Borel lifting. Define the $1$-cocycle $\om : \Gamma \times G/H \recht H$ by
$$g \theta(xH) = \theta(gxH) \om(g,xH)$$
whenever $g \in \Gamma$ and $xH \in G/H$. Assume that $\om$ is cohomologous to a group morphism $\delta : \Gamma \recht H$. This means that we can choose the lifting $\theta$ such that $g \theta(xH) = \theta(gxH) \delta(g)$ for $g \in \Gamma$ and almost all $xH \in G/H$.

The image of any group morphism $\SL(n,\Z) \recht \SL(n-1,\R) \ltimes \R^{n-1}$ is finite (see \cite[Theorem 6]{steinberg} for an elementary argument). So, we have found a finite index subgroup $\Gamma_0 \subset \SL(n,\Z)$ and a measurable map $\theta : \R^n \recht \SL(n,\R)$ such that $\theta(gx) = g \theta(x)$ for all $g \in \Gamma_0$ and almost all $x \in \R^n$. It follows that the map $(x,y) \mapsto \theta(x)^{-1} \theta(y)$ is invariant under the diagonal $\Gamma_0$-action, which is ergodic by Lemma \ref{lemma.ergodic-T}. Hence, $\theta$ is essentially constant, which is a contradiction with the formula $\theta(gx) = g \theta(x)$.
\end{proof}

\subsection{Proof of Theorems \ref{thm.OEsuperrigid-1} and \ref{thm.OEsuperrigid-2}}

We will deduce Theorems  \ref{thm.OEsuperrigid-1} and \ref{thm.OEsuperrigid-2} from the general Theorem \ref{thm.Haction} below, dealing with arbitrary actions of the form $\Gamma \actson \M_{n,k}(\R)/H$, where $\Gamma$ is a lattice and $H < \GL(k,\R)$ a closed subgroup. First of all, observe that these actions are essentially free and ergodic.

\begin{lemma} \label{lemma.measures-and-co}
Let $n > k$ and $\Gamma < \SL(n,\R)$ any lattice. Let $H \subset \GL(k,\R)$ be a closed subgroup. If $(-1,-1) \in \Gamma \times H$, put $\Gamma_0 = \Gamma/\{\pm 1\}$, otherwise put $\Gamma = \Gamma_0$.

The action $\Gamma_0 \actson \M_{n,k}(\R)/H$ is essentially free and ergodic. It never admits an invariant probability measure. It admits an infinite invariant measure if and only if $H$ is unimodular and satisfies $\det g = \pm 1$ for all $g \in H$.
\end{lemma}
\begin{proof}
By Lemma \ref{lemma.ergodic-T}, $\Gamma \actson \M_{n,k}(\R)$ is ergodic, because $k < n$. A fortiori, $\Gamma \actson \M_{n,k}(\R)/H$ is ergodic.

Denote $V = \lspan\{e_1,\ldots,e_k\} \subset \R^n$ and define the closed subgroup $H_1 < \SL(n,\R)$ by
$$H_0 = \{g \in \SL(n,\R) \mid g V = V \;\;\text{and}\;\; g|_V \in H \;\}\; .$$
We can identify $\Gamma_0 \actson \M_{n,k}(\R)/H$ with $\Gamma \actson \SL(n,\R)/H_1$. From this description, essential freeness follows. Also the statement about invariant measures follows, because $H_1$ is unimodular if and only if $H$ is unimodular and satisfies $\det g = \pm 1$ for all $g \in H$.
\end{proof}

\begin{theorem} \label{thm.Haction}
Let $H \subset \GL(k,\R)$ be a closed subgroup and $\Lambda \actson (Y,\eta)$ any essentially free, ergodic, non-singular action of the countable group $\Lambda$. Suppose $n \geq 4k+1$.

{\bf Case $-1 \in H$.} Let $\Gamma \subset \PSL(n,\R)$ be a lattice and put $\Gammatil := \{\pm 1\} \cdot \Gamma \subset \GL(n,\R)$. The actions $\Gamma \actson \M_{n,k}(\R)/H$ and $\Lambda \actson Y$ are SOE if and only if $\Lambda \actson Y$ is conjugate to an induction of one of the following actions:
    \begin{enumerate}
    \item $\Gamma \times H/N \actson \M_{n,k}(\R) / N$, where $N \lhd H$ is an open normal subgroup with $-1 \in N$,
    \item $\displaystyle \frac{\Gammatil \times H/N}{\{\pm (1,1)\}} \actson \M_{n,k}(\R) / N$, where $N \lhd H$ is an open normal subgroup with $-1 \not\in N$.
    \end{enumerate}

{\bf Case $-1 \not\in H$.} Let $\Gamma \subset \SL(n,\R)$ be a lattice. The actions $\Gamma \actson \M_{n,k}(\R)/H$ and $\Lambda \actson Y$ are SOE if and only if $\Lambda \actson Y$ is conjugate to an induction of one of the following actions:
    \begin{enumerate}
    \item $\Gamma \times H/N \actson \M_{n,k}(\R)/N$, where $N \lhd H$ is an open normal subgroup,
    \item (only when $-1 \in \Gamma$) $\displaystyle \frac{\Gamma}{\{\pm 1\}} \times \frac{H}{N} \actson \M_{n,k}(\R) / \bigl(\{\pm 1\} \cdot N\bigr)$, where $N \lhd H$ is an open normal subgroup.
    \end{enumerate}
\end{theorem}

\begin{remark}
Given a stable orbit equivalence $\Delta : \M_{n,k}(\R)/ H \recht Y$ between the actions $\Gamma \actson \M_{n,k}(\R)/H$ and $\Lambda \actson Y$, Theorem \ref{thm.Haction} provides a conjugacy $\Psi$ between one of the listed actions and $\Lambda_1 \actson Y_1$ such that $\Lambda \actson Y$ is induced from $\Lambda_1 \actson Y_1$. In fact, one moreover has $\Psi(\ox) \in \Lambda \cdot \Delta(\ox)$ for almost every $x \in \M_{n,k}(\R)$.
\end{remark}

If $\Delta : X_0 \recht Y_0$ is a SOE between $\Gamma \actson X$ and $\Lambda \actson Y$, we can (and will tacitly) extend $\Delta$ to a countable-to-one, measurable $\Delta : X \recht Y$ satisfying $\Delta(g \cdot x) \in \Lambda \cdot \Delta(x)$ for all $g \in \Gamma$ and almost all $x \in X$.

If $\Delta : X \recht Y$ is a SOE between the essentially free actions $\Gamma \actson X$ and $\Lambda \actson Y$, we get a \emph{Zimmer $1$-cocycle} $\om$ for the action $\Gamma \actson X$ with values in $\Lambda$, determined by the formula
$$\Delta(g \cdot x) = \om(g,x) \cdot \Delta(x)$$
almost everywhere. As a general principle, if the $1$-cocycle $\om$ is cohomologous to a group morphism $\Gamma \recht \Lambda$, the stable orbit equivalence is \lq essentially\rq\ given by a conjugacy of the actions: see \cite[Proposition 4.2.11]{Z2}, \cite[Lemma 4.7]{V-Bourbaki} and \cite[Theorem 1.8]{furman-on-popa}. In our framework of non p.m.p.\ actions, we again need such a principle: see Lemma \ref{lemma.OE} below.

The non-singular action $G \actson (X,\mu)$ of the l.c.s.c.\ group $G$ on the standard measure space $(X,\mu)$ is called \emph{essentially free and proper} if there exists a measurable map $\pi : X \recht G$ such that $\pi(g \cdot x) = g \pi(x)$ for almost all $(g,x) \in G \times X$. Equivalently, there exists a $^*$-isomorphism $\rL^\infty(X,\mu) \recht \rL^\infty(G) \ovt \rL^\infty(X,\mu)^G$ conjugating the natural $G$-actions.

\begin{lemma} \label{lemma.OE}
Let $G$ be a l.c.s.c.\ group and $G \actson[\si] (X,\mu)$ a non-singular, essentially free, ergodic action.  Assume that $\si$ is $\Ufin$-cocycle superrigid. Let $N \lhd G$ be an open normal subgroup such that the restricted action $\si|_N$ is proper. Let $\Lambda \actson (Y,\eta)$ be an essentially free, ergodic, non-singular action.

If $\Delta : X/N \recht Y$ a SOE between the actions $G/N \actson X/N$ and $\Lambda \actson Y$, there exists
\begin{itemize}
    \item a subgroup $\Lambda_1 < \Lambda$ and a non-negligible $Y_1 \subset Y$ such that $\Lambda \actson Y$ is induced from $\Lambda_1 \actson Y_1$~;
    \item an open normal subgroup $N_1 \lhd G$ such that $\si|_{N_1}$ is proper~;
\end{itemize}
such that the actions $G/N_1 \actson X/N_1$ and $\Lambda_1 \actson Y_1$ are conjugate through the non-singular isomorphism $\Psi : X/N_1 \recht Y_1$ and the group isomorphism $\delta : G/N_1 \recht \Lambda_1$. Furthermore, $\Delta(\ox) \in \Lambda \cdot \Psi(\ox)$ for almost all $x \in X$.
\end{lemma}
\begin{proof}
Let $\Delta : X/N \recht Y$ be a SOE. By cocycle superrigidity of $\si$, take a measurable map $\vphi : X \recht \Lambda$ such that, writing $\Theta(x) = \vphi(x)^{-1} \cdot \Delta(\ox)$, we have $\Theta(g \cdot x) = \delta(g) \cdot \Theta(x)$ almost everywhere, where $\delta : G \recht \Lambda$ is a continuous group morphism. Put $N_1 = \Ker \delta$. So, $N_1$ is an open normal subgroup of $G$. Then, $N \cap N_1$ is still open in $G$ and we consider $X/(N \cap N_1)$ with the quotient map $\pi : X/(N \cap N_1) \recht X/N$. It follows that we can view $\Theta$ as a measurable map $\Theta : X/(N \cap N_1) \recht Y$ such that $\Delta(\pi(x)) \in \Lambda \cdot \Theta(x)$ for almost all $x \in X / (N \cap N_1)$ and $\Theta(g \cdot x) = \delta(g) \cdot \Theta(x)$ almost everywhere.

Using the facts that the countable group $N / (N \cap N_1)$ acts freely and properly on $X/(N \cap N_1)$, that $\Lambda$ is countable and that $\Delta : X/N \recht Y$ is locally a non-singular isomorphism, it follows that $X/(N \cap N_1)$ can be partitioned in a sequence of non-negligible subsets $(\cU_n)_n$, such that for every $n$, $\Theta|_{\cU_n}$ is a non-singular isomorphism between $\cU_n$ and a non-negligible subset of $Y$. But then, for every non-trivial element $g \in N_1 / (N \cap N_1)$ and every $n$, we conclude that $g \cdot \cU_n \cap \cU_n$ has measure zero. It follows that $N_1 / (N \cap N_1)$ acts freely and properly on $X/(N \cap N_1)$. So, $N_1$ acts freely and properly on $X$. Hence, we can form the quotient space $X/N_1$ and view $\Theta$ as a measurable map $\Theta : X/N_1 \recht Y$ such that $\Delta(\ox) \in \Lambda \cdot \Theta(\ox)$ for almost all $x \in X$ and $\Theta(\og \cdot \ox) = \delta(\og) \cdot \Theta(\ox)$ almost everywhere. Now, $\delta : G/N_1 \recht \Lambda$ is an injective group morphism. Still, $X/N_1$ can be partitioned into a sequence of non-negligible subsets $(\cU_n)$ such that $\Theta|_{\cU_n}$ is a non-singular isomorphism between $\cU_n$ and $\cV_n \subset Y$.

We claim that $\cV_n \cap \cV_m$ has measure zero for every $n \neq m$. If this is not the case, take $\cW \subset \cU_n$ and $\cW' \subset \cU_m$ non-negligible and a non-singular isomorphism $\rho : \cW \recht \cW'$ such that $\Theta(\rho(x)) = \Theta(x)$ for $x \in \cW$. Since $\Delta$ is a SOE, $\rho(x) \in (G/N_1) \cdot x$ for almost all $x \in \cW$. Hence, making $\cW$ smaller but still non-negligible, we may assume that $\rho(x) = \og \cdot x$ for all $x \in \cW$ and some $\og \in G/N_1$. Since $\cW \cap \cW'$ has measure zero and the action of $G/N_1$ on $X/N_1$ is essentially free, we get $\og \neq e$. But also, $\delta(\og) \cdot \Theta(x) = \Theta(x)$ for almost all $x \in \cW$. This is a contradiction with the injectivity of $\delta$ and the essential freeness of $\Lambda \actson Y$. This proves the claim and we have found that $\Theta$ is a non-singular isomorphism between $X/N_1$ and a non-negligible subset $Y_1 \subset Y$. Set $\Lambda_1 = \delta(G/N_1)$.

It remains to prove that $\Lambda \actson Y$ is induced from $\Lambda_1 \actson Y_1$. So, let $h \in \Lambda$ and assume that $h \cdot Y_1 \cap Y_1$ is non-negligible. We have to prove that $h \in \Lambda_1$. By our assumption, take $\cW,\cW' \subset X/N_1$ non-negligible and a non-singular isomorphism $\rho : \cW \recht \cW'$ such that $h \cdot \Theta(x) = \Theta(\rho(x))$ for all $x \in \cW$. Since $\Delta$ was a SOE, we can make $\cW$ smaller but still non-negligible and assume that $\rho(x) = \og \cdot x$ for all $x \in \cW$ and some $\og \in G/N_1$. But then, $h \cdot \Theta(x) = \delta(\og) \cdot \Theta(x)$ for almost all $x \in \cW$. Since $\Lambda \actson Y$ is essentially free, it follows that $h = \delta(\og) \in \Lambda_1$.
\end{proof}

\begin{proof}[Proof of Theorem \ref{thm.Haction}]
We only prove the case $-1 \in H$, the case $-1 \not\in H$ being analogous. By Theorem \ref{thm.cocycle-superrigid-concrete}, the action of $G:=(\Gammatil \times H)/\{\pm (1,1)\}$ on $\M_{n,k}(\R)$ is $\Ufin$-cocycle superrigid. By Lemma \ref{lemma.OE}, we only have to prove that the following subgroups of $\Gammatil \times H$ are the only open normal subgroups containing $(-1,-1)$ and acting properly on $\M_{n,k}(\R)$.
\begin{itemize}
\item $\{\pm 1\} \times N$, where $N \lhd H$ is an open normal subgroup with $-1 \in N$.
\item $(\{1\} \times N) \cup (\{-1\} \times -N)$, where $N \lhd H$ is an open normal subgroup with $-1 \not\in N$.
\end{itemize}
So, let $N \lhd (\Gammatil \times H)$ be a closed normal subgroup acting properly on $\M_{n,k}(\R)$. It is sufficient to prove that $N \subset \{\pm 1\} \times H$. Suppose that $N \not\subset \{\pm 1\} \times H$ and take $(g,h) \in N$ with $g \neq \pm 1$. Take $k \in \Gammatil$ such that the commutator $t := kgk^{-1} g^{-1} \neq \pm 1$. It follows that $(t,1) \in N$. By Margulis' normal subgroup theorem \cite{margulis}, we have $\Gamma_0 \times \{1\} \subset N$ for some finite index subgroup $\Gamma_0 < \Gammatil$. By Lemma \ref{lemma.ergodic-T}, $\Gamma_0$ acts ergodically on $\M_{n,k}(\R)$, contradicting the properness of $N \actson \M_{n,k}(\R)$.
\end{proof}

\begin{proof}[Proof of Theorem \ref{thm.OEsuperrigid-1}]
Statement 1 follows immediately from Theorem \ref{thm.Haction}.

We now prove statement 2. We claim that the following are the only normal subgroups $N$ of $\SL(n,\Z) \ltimes \Z^n$ that act properly on $\R^n$~:
\begin{itemize}
\item $N=\{e\}$,
\item $N = \lambda \Z^n$ for some $\lambda \in \N \setminus \{0\}$,
\item (only when $n$ is even) $N = \{\pm 1\} \ltimes \lambda \Z^n$ for $\lambda \in \{1,2\}$.
\end{itemize}
Statement 2 of Theorem \ref{thm.OEsuperrigid-1} then follows from Lemma \ref{lemma.OE} and Theorem \ref{thm.cocycle-superrigid-concrete}, where the $\Ufin$-cocycle superrigidity of the affine action $\SL(n,\Z) \ltimes \Z^n \actson \R^n$ was established.

So, let $N \lhd \SL(n,\Z) \ltimes \Z^n$ be a normal subgroup acting properly on $\R^n$. Suppose first that $N \not\subset \Z^n$. Taking the commutator of $(g,x) \in N$ with $g \neq 1$ and an arbitrary $(1,y)$, $y \in \Z^n$, it follows that $H:=N \cap \Z^n \neq \{0\}$. Hence, $H$ is a non-zero, globally $\SL(n,\Z)$-invariant subgroup of $\Z^n$. So, $H = \lambda \Z^n$ for some $\lambda \in \N \setminus \{0\}$. If $N \not\subset \{\pm 1\} \ltimes \Z^n$, it would follow that $N$ has finite index in $\SL(n,\Z) \ltimes \Z^n$, contradicting the properness of $N \actson \R^n$. So, we have shown that in all cases $N \subset \{\pm 1\} \ltimes \Z^n$. It is now straightforward to deduce the above list of possibilities for $N$.
\end{proof}

\begin{proof}[Proof of Theorem \ref{thm.OEsuperrigid-2}]
This theorem is a special case of Theorem \ref{thm.Haction}.
\end{proof}

\subsection{Describing all $1$-cocycles of quotient actions}

Finally, cocycle superrigidity of $G \actson (X,\mu)$ allows to describe all $1$-cocycles for $G/N \actson X/N$ when $N \lhd G$ is a closed normal subgroup of $G$ acting essentially freely and properly on $X$. We start with the following proposition, closely related to \cite[Lemma 5.3]{PV2}, and illustrate it with two examples.

\begin{proposition} \label{prop.quotient}
Let $G$ be a l.c.s.c.\ group and $G \actson[\si] (X,\mu)$ a non-singular action. Let $N \lhd G$ be a closed, normal subgroup such that the restriction $\si|_N$ is essentially free and proper. Assume that $\si$ is $\Ufin$-cocycle superrigid.

Choose a measurable map $\pi : X \recht N$ satisfying $\pi(g \cdot x) = g \pi(x)$ for almost all $(g,x) \in N \times X$. Denote by $g \mapsto \og$ and $x \mapsto \ox$ the quotient maps $G \recht G/N$, resp.\ $X \recht X/N$. Then,
$$\om : \frac{G}{N} \times \frac{X}{N} \recht G : \om(\og,\ox) = \pi(g \cdot x)^{-1} g \pi(x)$$
is a well-defined $1$-cocycle.

Every $1$-cocycle for the action $G/N \actson X/N$ with values in a Polish group of finite type $\cG$ is cohomologous to $\delta \circ \om$ for a continuous morphism $\delta : G \recht \cG$. If the diagonal action $G \actson X \times X$ is ergodic, $\delta$ is uniquely determined up to conjugacy by an element of $g \in \cG$.
\end{proposition}
\begin{proof}
Let $\Omega : G/N \times X/N \recht \cG$ be a $1$-cocycle with values in the Polish group of finite type $\cG$. From cocycle superrigidity of $\si$, let $\vphi : X \recht \cG$ be a measurable map and $\delta : G \recht \cG$ a continuous group morphism such that
$$\Omega(\og,\ox) = \vphi(g \cdot x)^{-1} \delta(g) \vphi(x)$$
almost everywhere. Replacing $g$ by $hg$, $h \in N$, it follows that $\vphi(hg \cdot x)^{-1} \delta(h) = \vphi(g \cdot x)^{-1}$ and hence, $\vphi(h \cdot x) = \delta(h) \vphi(x)$ almost everywhere. So, we can define $\Psi(\ox) = \delta(\pi(x))^{-1} \vphi(x)$. By construction, $\Psi$ makes $\Omega$ cohomologous to $\delta \circ \omega$. The uniqueness of $\delta$ follows directly from Lemma \ref{lemma.tata}.
\end{proof}

\begin{example} \label{ex.cocycles-torus}
Let $\Gamma \subset \SL(n,\Z)$ be a finite index subgroup and consider the action $\Gamma \actson \T^n$.
    \begin{itemize}
    \item Choosing a measurable map $p : \R^n \recht \Z^n$ such that $p(x+y) = x + p(y)$ for all $x \in \Z^n$ and almost all $y \in \R^n$, the formula $$\om : \Gamma \times \R^n/\Z^n \recht \Gamma \ltimes \Z^n : \om(g,\ox) = p(g \cdot x)^{-1} g p(x)$$ defines a $1$-cocycle for $\Gamma \actson \T^n$ with values in $\Gamma \ltimes \Z^n$.
    \item Every $1$-cocycle with values in a Polish group of finite type $\cG$ is cohomologous with $\delta \circ \om$ for a group morphism $\delta : \Gamma \ltimes \Z^n \recht \cG$, uniquely determined up to conjugacy by an element in $\cG$.
    \end{itemize}
\end{example}

Note that by Zimmer's cocycle superrigidity theorem \cite[Theorem 5.2.5]{Z2},
any Zariski dense $1$-cocycle for any ergodic p.m.p.\ action $\SL(n,\Z) \actson
(X,\mu)$, $n \geq 3$, taking values in a connected simple real
algebraic non-compact center free group, is
cohomologous to a group morphism. For the specific action $\SL(n,\Z)
\actson \R^n/\Z^n$, $n \geq 5$, the previous example provides an explicit
description of all $1$-cocycles of $\SL(n,\Z) \actson \R^n / \Z^n$
with values in an arbitrary Polish group of finite type.

\begin{example} \label{ex.cocycles-flag}
Let $\Gamma \subset \PSL(n,\R)$ be a lattice and $X$ the real flag manifold of signature $(d_1,\ldots,d_l,n)$ with $n \geq 4 d_l +1$. We obtain as follows all $1$-cocycles for the action $\Gamma \actson X$ with values in a Polish group of finite type.
    \begin{itemize}
    \item Identify $X = \M_{n,k}(\R) /H$ and choose a measurable map $p : \M_{n,k}(\R) \recht H$ satisfying $p(A g) = p(A) g$ for almost all $A \in \M_{n,k}(\R)$, $g \in H$. Let $\Gammatil = \{\pm 1\} \cdot \Gamma$ be the double cover of $\Gamma$ in $\GL(n,\R)$ and define $G := (\Gammatil \times H)/\{\pm (1,1)\}$. The formula $$\om : \Gamma \times X \recht G : \om(\og,\ox) = (g, p(g x) p(x)^{-1}) \mod \{\pm (1,1)\}$$
        defines a $1$-cocycle with values in $G$. Here $\Gammatil \recht \Gamma : g \mapsto \og$ and $\M_{n,k}(\R) \recht X : x \mapsto \ox$ are the quotient maps.
    \item Every $1$-cocycle with values in a Polish group of finite type $\cG$ is cohomologous with $\delta \circ \om$ for a continuous group morphism $\delta : G \recht \cG$, uniquely determined up to conjugacy by an element in $\cG$.
    \end{itemize}
\end{example}

\section{Classification up to orbit equivalence}

Combining the results of \cite{F2} with Theorem \ref{thm.Haction}, we classify up to stable orbit equivalence, the linear lattice actions $\Gamma \actson \R^n$, as well as the natural lattice actions on flag manifolds. At the same time, we compute the outer automorphism group of the associated orbit equivalence relations.

We start with the following elementary lemma.

\begin{lemma}\label{lemma.notinduced}
Let $\Gamma \actson (X,\mu)$ be a non-singular action of the countable group $\Gamma$ and assume that the diagonal action $\Gamma \actson X \times X$ is ergodic. Then, $\Gamma \actson X$ is not induced, i.e.\ if $\Gamma \actson X$ is induced from $\Gamma_1 \actson X_1$, then $\Gamma_1 = \Gamma$ and $\mu(X \setminus X_1) = 0$.
\end{lemma}
\begin{proof}
Assume that $\Gamma \actson X$ is induced from $\Gamma_1 \actson X_1$. So, we find a quotient map $\pi : X \recht \Gamma/\Gamma_1$ satisfying $\pi(g \cdot x) = g \pi(x)$ almost everywhere and $X_1 = \pi^{-1}(e \Gamma_1)$. Hence, the subset $\{(x,y) \in X \times X \mid \pi(x) = \pi(y)\}$ is non-negligible and $\Gamma$-invariant. By the ergodicity of $\Gamma \actson X \times X$, it follows that $\pi(x) = \pi(y)$ for almost all $(x,y) \in X \times X$. This means that $\Gamma_1 = \Gamma$ and $\mu(X \setminus X_1) = 0$.
\end{proof}

If $\cR$ is a II$_1$ equivalence relation on $(X,\mu)$, we denote by $[\cR]$ the \emph{full group} of $\cR$ consisting of non-singular automorphisms $\Delta : X \recht X$ satisfying $(x,\Delta(x)) \in \cR$ for almost every $x \in X$. Then, $[\cR]$ is a normal subgroup of the automorphism group $\Aut(\cR)$ of $\cR$. The quotient group is denoted by $\Out(\cR)$ and called the \emph{outer automorphism group} of $\cR$. The \emph{full pseudogroup} of $\cR$ is denoted by $[[\cR]]$ and consists of non-singular partial isomorphisms $\phi : X_0 \subset X \recht X_1 \subset X$ satisfying $(x,\phi(x)) \in \cR$ for almost every $x \in X_0$. We denote $X_0 = D(\phi)$ and $X_1 = R(\phi)$.

\begin{theorem}\label{thm.classif-linear}
Let $n \geq 5$ and $\Gamma \subset \SL(n,\R)$ a lattice. Let $n' \geq 2$ and $\Gamma' \subset \SL(n',\R)$ a lattice. If the non-singular isomorphism $\Delta : X_1 \subset \R^n \recht X_1' \subset \R^{n'}$ is a SOE between $\Gamma \actson \R^n$ and $\Gamma' \actson \R^{n'}$, then
\begin{itemize}
\item $n=n'$ and there exists $A \in \GL(n,\R)$ such that $\Gamma' = A \Gamma A^{-1}$,
\item there exists $\phi \in [[\cR(\Gamma' \actson \R^{n'})]]$ with $R(\phi) = X_1'$ such that $\Delta(x) = \phi(A(x))$ for almost every $x \in X_1$.
\end{itemize}
In particular, $\Out(\cR(\Gamma \actson \R^n)) = \cN_{\GL(n,\R)}(\Gamma)/\Gamma$.
\end{theorem}

\begin{proof}
Let $\Delta : X_1 \subset \R^n \recht X_1' \subset \R^{n'}$ be a SOE between $\Gamma \actson \R^n$ and $\Gamma' \actson \R^{n'}$. Since for $n'=2$, the equivalence relation $\cR(\Gamma' \actson \R^{n'})$ is hyperfinite, we have $n' \geq 3$ and hence the diagonal action $\Gamma' \actson \R^{n'} \times \R^{n'}$ is ergodic. By Lemma \ref{lemma.notinduced}, $\Gamma' \actson \R^{n'}$ is not an induced action.
By Theorem \ref{thm.OEsuperrigid-1}, the action $\Gamma' \actson \R^{n'}$ is conjugate with either $\Gamma \actson \R^n$ or, in case $n$ is even, $\Gamma / \{\pm 1\} \actson \R^n / \{\pm 1\}$. By Mostow rigidity, the latter is impossible since $\Gamma' \not\cong \Gamma/\{\pm 1\}$.
In the former case, we already conclude $n=n'$ and we have found a non-singular isomorphism $\Theta : \R^n \recht \R^n$ and a group isomorphism $\delta : \Gamma \recht \Gamma'$ satisfying
\begin{itemize}
\item $\Theta(g \cdot x) = \delta(g) \cdot \Theta(x)$,
\item $\Delta(x) \in \Gamma' \cdot \Theta(x)$,
\end{itemize}
for all $g \in \Gamma$ and almost all $x \in \R^n$. Denoting by $B\trans$ the transpose of the matrix $B$, by Mostow rigidity, we find $A \in \GL(n,\R)$ such that a) $\delta(g) = AgA^{-1}$ for all $g \in \Gamma$, or b) $\delta(g) = A (g\trans)^{-1} A^{-1}$ for all $g \in \Gamma$.

Define the subgroup $H \subset \SL(n,\R)$ consisting of matrices $g$ with $g e_1 = e_1$ and identify $\R^n = \SL(n,\R)/H$. In case b), we would get a conjugacy between the $\Gamma$-actions on $\SL(n,\R)/H$ and $\SL(n,\R)/H\trans$, which is ruled out by \cite[Theorem D]{F2}. In case a), multiplying $A$ by a non-zero scalar if necessary, \cite[Theorem D]{F2} implies that $\Theta(x) = A x$ for almost all $x \in \R^n$.

Defining the partial isomorphism $\phi := \Delta \circ A^{-1}$ with $R(\phi) = X_1'$, it follows that $\phi \in [[\cR(\Gamma' \actson \R^{n'})]]$ and that $\Delta(x) = \phi(A(x))$ for almost every $x \in X_1$.
\end{proof}

Let $X$ be the real flag manifold with signature $d:=(d_1,\ldots,d_l,n)$. Denote
$$d\trans := (n-d_l,n - d_{l-1},\ldots, n-d_1,n) \; .$$
If $X'$ is the real flag manifold with signature $d\trans$, there is a natural diffeomorphism $X \recht X' : x \mapsto \overline{x}$ satisfying $\overline{g \cdot x} = (g\trans)^{-1} \cdot \overline{x}$ for all $x \in X$, $g \in \SL(n,\R)$.

\begin{theorem}\label{thm.classif-flag}
Let $\Gamma \subset \PSL(n,\R)$ be a lattice and $X$ the real flag manifold with signature $d:=(d_1,\ldots,d_l,n)$. Assume that $n \geq 4d_l+1$. Let $\Gamma' \subset \PSL(n',\R)$ be a lattice and $X'$ the real flag manifold with signature $d':=(d'_1,\ldots,d'_{l'},n')$.

If the non-singular isomorphism $\Delta : X_1 \subset X \recht X_1' \subset X'$ is a SOE between $\Gamma \actson X$ and $\Gamma' \actson X'$, then $n=n'$ and there exists $A \in \PGL(n,\R)$, $\phi \in [[\cR(\Gamma' \actson X')]]$ with $R(\phi) = X_1'$ such that
\begin{itemize}
\item either, $d'=d$, $\Gamma' = A \Gamma A^{-1}$, $\Delta(x) = \phi(A(x))$ for almost every $x \in X_1$,
\item or, $d'=d\trans$, $\Gamma' = A \Gamma\trans A^{-1}$, $\Delta(x) = \phi(A(\overline{x}))$ for almost every $x \in X_1$.
\end{itemize}
In particular, $\Out(\cR(\Gamma \actson X)) = \cN_{\PGL(n,\R)}(\Gamma)/\Gamma$.
\end{theorem}
\begin{proof}
Let $\Delta : X_1 \subset X \recht X_1' \subset X'$ be a SOE between $\Gamma \actson X$ and $\Gamma' \actson X'$.

We first prove that the diagonal action $\Gamma' \actson X' \times X'$ is ergodic. If $Y$ denotes the flag manifold of signature $(1,2,\ldots,n')$, the action $\Gamma' \actson X'$ is a quotient of the action $\Gamma \actson Y$. So, it suffices to prove ergodicity of $\Gamma' \actson Y \times Y$. Denoting by $D \subset \SL(n',\R)$ the subgroup of diagonal matrices, $\Gamma' \actson Y \times Y$ can be identified with $\Gamma' \actson \SL(n',\R) / D$, which follows ergodic because $D \actson \SL(n',\R)/\Gamma'$ is ergodic by Moore's ergodicity theorem. So, by Lemma \ref{lemma.notinduced}, $\Gamma' \actson X'$ is not an induced action.

Since $\Gamma'$ has trivial center, Theorem \ref{thm.OEsuperrigid-2} yields a group isomorphism $\delta : \Gamma \recht \Gamma'$ and a non-singular isomorphism $\Theta : X \recht X'$ satisfying $\Theta(g \cdot x) = \delta(g) \cdot \Theta(x)$ and $\Delta(x) \in \Gamma' \cdot \Theta(x)$ almost everywhere. By Mostow rigidity, $n = n'$ and there exists $A \in \PGL(n,\R)$ such that either $\delta(g) = A g A^{-1}$ or $\delta(g) = A (g\trans)^{-1} A^{-1}$ for all $g \in \Gamma$.

Set $k_1 = d_1$ and $k_i = d_i - d_{i-1}$ for $2 \leq i \leq l$. Put $k=d_l$ and define the closed subgroups $H$ and $H_1$ of $\GL(k,\R)$ by
\begin{align*}
H & := \left\{ \begin{pmatrix} A_{11} & * & \cdots & * \\ 0 & A_{22} & \cdots & * \\
\vdots & \vdots & \ddots & \vdots \\
0 & 0 & \cdots & A_{ll} \end{pmatrix} \; \Big| \; A_{ii} \in \GL(k_i,\R) \;\;\text{and}\;\; \det(A_{ii}) \neq 0 \right\} \quad , \vspace{0.5ex}\\
H_1 & := \left\{ \begin{pmatrix} A_{11} & * & \cdots & * \\ 0 & A_{22} & \cdots & * \\
\vdots & \vdots & \ddots & \vdots \\
0 & 0 & \cdots & A_{ll} \end{pmatrix} \; \Big| \; A_{ii} \in \GL(k_i,\R) \;\;\text{and}\;\; \det(A_{ii}) = \pm 1 \right\} \; .
\end{align*}
We identify $H/H_1 = \R_+^k$. From now on, we write $X$ as $\M_{n,k}(\R)/H$. Define analogously the subgroups $H'$, $H_1'$ of $\GL(k',\R)$ and write $X'$ as $\M_{n,k}(\R)/H'$

Whenever $N \lhd H$ is a closed normal subgroup containing $-1$, the quotient morphism $H \recht H/N$ gives rise, as in Example \ref{ex.cocycles-flag}, to a $1$-cocycle $\om_{N}$ for the action $\Gamma \actson X$ with values in $H/N$ such that the action $\Gamma \actson \M_{n,k}(\R) / N$ can be identified with the action $\Gamma \actson X \times H/N$ given by
$$g \cdot (x,h) = (g \cdot x, \om_N(g,x)h) \; .$$
Since $\Gamma \actson \M_{n,k}(\R) /N$ is ergodic, the $1$-cocycle $\om_N$ cannot be cohomologous to a $1$-cocycle taking values in a proper closed subgroup of $H/N$. We similarly define the $1$-cocycles $\om_{N'}$ for $\Gamma' \actson X'$.

Since $\delta,\Theta$ conjugate the actions $\Gamma \actson X$ and $\Gamma' \actson X'$, the map $\mu(g,x) = \om_{H_1'}(\delta(g),\Theta(x))$ defines a $1$-cocycle for $\Gamma \actson X$, with values in $\R_+^{k'}$ and with the property of not being cohomologous to a $1$-cocycle taking values in a proper closed subgroup of $\R_+^{k'}$. We now apply Example \ref{ex.cocycles-flag}, describing all $\Ufin$-valued $1$-cocycles for $\Gamma \actson X$, and note that $\R_+^{k'}$ belongs to $\Ufin$. Every group morphism $\Gamma \recht \R_+^{k'}$ is trivial and every continuous group morphism $H \recht \R_+^{k'}$ is trivial on $H_1$. So, we find a continuous group morphism $\rho : \R_+^k \recht \R_+^{k'}$ such that $\mu$ is cohomologous to $\rho \circ \om_{H_1}$. Since $\mu$ cannot be cohomologous to a $1$-cocycle taking values in a proper closed subgroup of $\R_+^{k'}$, it follows that $\rho$ is onto.

Altogether, we find a closed normal subgroup $N \lhd H$, containing $H_1$, and a continuous isomorphism $\rho : H/N \recht H'/H_1'$ such that $\mu$ is cohomologous to $\rho \circ \om_N$. It follows that there exists a non-singular isomorphism $\Theta_1 : \M_{n,k}(\R)/N \recht \M_{n',k'}(\R)/H_1'$ satisfying $\Theta_1(g \cdot x) = \delta(g) \cdot \Theta_1(x)$ and $\Theta_1(x)H' = \Theta(xH)$ almost everywhere.

Since the action of $\Gamma'$ on $\M_{n',k'}(\R) / H_1'$ is infinite measure preserving, Lemma \ref{lemma.measures-and-co} implies that $N = H_1$. We saw already that either $\delta(g) = A g A^{-1}$ or $\delta(g) = A (g\trans)^{-1} A^{-1}$ for all $g \in \Gamma$. In the former case, \cite[Theorem D]{F2} implies that $d'=d$ and that there exists $B \in H$ such that $\Theta_1(x) = A x B$ for almost every $x \in \M_{n,k}(\R)/H_1$. It follows that $\Theta(x) = A(x)$ for almost every $x \in X$. In the latter case, we prove analogously that $d'=d\trans$ and $\Theta(x) = A(\overline{x})$.
\end{proof}

\section{Implementation by group actions}

In \cite[page 292]{FM}, the question is raised whether every II$_1$ equivalence relation can be implemented by an essentially free action of a countable group. This question has been settled in the negative in \cite[Theorem D]{F1}. In Proposition \ref{prop.notimplement} below, we give examples of II$_1$ equivalence relations $\cR$ on $(X,\mu)$ with the following much stronger property: whenever $\Lambda \actson (Y,\eta)$ is an essentially free, non-singular action and $\Delta : X \recht Y$ is a measurable map satisfying $\Delta(x) \in \Lambda \cdot \Delta(y)$ for almost all $(x,y) \in \cR$, then there exists $y_0 \in Y$ such that $\Delta(x) \in \Lambda \cdot y_0$ for almost all $x \in X$.

Among other examples, \cite[Theorem D]{F1} proves that the restriction of the orbit equivalence relation $\SL(n,\Z) \actson \R^n / \Z^n$ to a subset of irrational measure, provides a II$_1$ equivalence relation that cannot be implemented by a free action of a group. By \cite[Theorem 0.3]{Pcocycle}, if $\Gamma \acts [0,1]^\Gamma$ is the Bernoulli action of a property (T) group $\Gamma$ without finite normal subgroups, the restriction of its orbit equivalence relation to any subset of measure strictly between $0$ and $1$ is unimplementable by a free action.

But, \cite[Theorem D]{F1} also provides examples of II$_1$ equivalence relations $\cR$ such that none of the amplifications $\cR^t$, $t > 0$, can be implemented by a free action. These equivalence relations are constructed using the following method. Suppose that $G$ is a l.c.s.c.\ unimodular group and $G \actson (X,\mu)$ an essentially free, properly ergodic, p.m.p.\ action. There exists a Borel set $Y \subset X$, a probability measure $\eta$ on $Y$ and a neighborhood $\cU$ of $e$ in $G$ that we equip with a multiple of the Haar measure, such that $\cU \times Y \recht X : (g,y) \mapsto g \cdot y$ provides a measure preserving isomorphism of $\cU \times Y$ onto a non-negligible subset of $X$. The restriction of the orbit equivalence relation of $G \actson (X,\mu)$ to $Y$ is a II$_1$ equivalence relation on $(Y,\eta)$. One calls $Y \subset X$ a \emph{measurable cross-section} for $G \actson (X,\mu)$. A different choice of measurable cross-section yields a stably isomorphic II$_1$ equivalence relation.

By Theorem \ref{thm.equiv-T}, the restriction of the orbit equivalence relation of $\SL(n,\Z) \actson \R^n$ to a subset of finite measure, provides other examples of II$_1$ equivalence relations $\cR$ such that none of the finite amplifications can be implemented by a free action. In fact, one can show that $\cR$ arises as the measurable cross-section for the action of $\SL(n-1,\R) \ltimes \R^{n-1}$ on $\SL(n,\R)/\SL(n,\Z)$. Nevertheless, this example is not covered by \cite[Theorem D]{F1}, since $\SL(n-1,\R) \ltimes \R^{n-1}$ is not semi-simple.

\begin{proposition}\label{prop.notimplement}
Let $G$ be a l.c.s.c.\ connected, unimodular group with normal closed subgroup $G_0$ having the relative property (T). Let $H_\R$ be a real Hilbert space and $\pi : G \recht \rO(H_\R)$ an orthogonal representation. Assume that $\pi$ is injective and that the restriction of $\pi$ to $G_0$ is weakly mixing (i.e.\ has no finite-dimensional invariant subspaces). Denote by $G \actson (X,\mu)$ the associated Gaussian action (see e.g.\ \cite[Section 2.7]{furman-on-popa}). Choose a measurable cross-section $X_1 \subset X$ and denote by $\cR$ the associated II$_1$ equivalence relation on $X_1$.

Whenever $\Lambda \actson (Y,\eta)$ is an essentially free, non-singular action and $\Delta : X_1 \recht Y$ is a measurable map satisfying $\Delta(x) \in \Lambda \cdot \Delta(y)$ for almost all $(x,y) \in \cR$, then there exists $y_0 \in Y$ such that $\Delta(x) \in \Lambda \cdot y_0$ for almost all $x \in X_1$.
\end{proposition}
\begin{proof}
Choose a measurable map $p : X \recht X_1$ such that $p(x) \in G \cdot x$ for almost all $x \in X$. Define the $1$-cocycle $\om : G \times X \recht \Lambda$ such that $\Delta(p(g \cdot x)) = \om(g,x) \cdot \Delta(p(x))$ almost everywhere. As observed in \cite{furman-on-popa}, Theorem 0.1 in \cite{Pcocycle} applies to $G \actson (X,\mu)$. Since $G$ is connected, every group morphism from $G$ to $\Lambda$ is trivial and we find a measurable map $\vphi : X \recht \Lambda$ such that $\om(g,x) = \vphi(g \cdot x)^{-1} \vphi(x)$. So, the map $x \mapsto \vphi(x) \cdot \Delta(p(x))$ is $G$-invariant and hence, essentially constant. We therefore find $y_0 \in Y$ such that $\Delta(p(x)) \in \Lambda \cdot y_0$ for almost all $x \in X$. This concludes the proof of the proposition.
\end{proof}

\end{document}